%% file: main.tex
\documentclass{article}

\PassOptionsToPackage{numbers, compress}{natbib}
\usepackage[preprint]{neurips_2026}
\makeatletter
\renewcommand{\@noticestring}{}
\makeatother

\usepackage[utf8]{inputenc} %
\usepackage[T1]{fontenc}    %
\usepackage{hyperref}       %
\usepackage{url}            %
\usepackage{booktabs}       %
\usepackage{amsfonts}       %
\usepackage{nicefrac}       %
\usepackage{microtype}      %
\usepackage{xcolor}         %

\usepackage{wrapfig}
\usepackage{widebar}

\usepackage{colortbl}
\usepackage{subfigure}

\usepackage{amsmath}
\usepackage{amsthm}
\usepackage{amssymb}
\usepackage{bbm}
\usepackage[capitalize,noabbrev]{cleveref}
\usepackage{refcount}       %

\usepackage{rl-notations}
\usepackage{thms}

\title{Revisiting Subgradient Dominance in Robust MDPs:\\Counterexamples, Hardness, and Sufficient Conditions}

\author{%
  Toshinori Kitamura \\
  University of Alberta \\
  \texttt{toshinor@ualberta.ca} \\
  \And
  Arnob Ghosh \\
  New Jersey Institute of Technology \\
  \texttt{arnob.ghosh@njit.edu} \\
  \And
  Alex Ayoub \\
  University of Alberta \\
  \texttt{aayoub@ualberta.ca} \\
  \And
  Thang D. Chu \\
  University of Alberta \\
  \texttt{thang@ualberta.ca} \\
  \And
  Csaba Szepesv\'{a}ri \\
  University of Alberta \\
  \texttt{szepesva@ualberta.ca} \\
}

\begin{document}

\maketitle

\begin{abstract}
Projected subgradient descent (PSD) has gained popularity for solving robust Markov decision processes (RMDPs) because it applies to a broader class of uncertainty sets than traditional dynamic programming.
  Existing work claims that RMDPs with a general compact uncertainty set satisfy the \emph{subgradient dominance property}, under which exact PSD finds an $\varepsilon$-optimal policy in a polynomial number of updates (e.g., \citet{wang2023policy}). We show that these claims are incorrect. Even when the uncertainty set has cardinality two, the RMDP objective \textbf{need not} be subgradient-dominant and can admit suboptimal strict local minima.
  Moreover, we prove that finding an $\varepsilon$-optimal policy can be NP-hard even in settings where subgradients are efficiently computable: {(\rm i)} finite transition uncertainty sets and {\rm(ii)} $sa$-rectangular finite transition uncertainty sets with finite non-rectangular cost uncertainty sets.
  Finally, we identify two conditions under which RMDPs do satisfy subgradient dominance: when, for each policy, either the worst-case transition kernel is unique or a robust action-value function exists.
\end{abstract}

\section{Introduction}

\looseness=-1
Dealing with environmental model uncertainty is crucial for practical decision-making problems \citep{taguchi1986introduction}. 
Robust Markov decision processes (RMDPs) provide a framework for designing policies that are robust to such uncertainty, where transition kernels and cost functions are chosen adversarially from an uncertainty set.
Recently, projected subgradient descent (PSD) has emerged as a popular approach for solving RMDPs \citep{wang2023policy,wang2022policy,kumar2024policy,kitamura2025near}.
PSD methods are particularly appealing for their broad applicability: they accommodate general uncertainty sets for which subgradients can be efficiently computed (e.g., finite sets), whereas traditional dynamic programming (DP) methods typically require stronger structural assumptions on the uncertainty set \citep{iyengar2005robust,wiesemann2013robust}.

\looseness=-1
Existing RMDP literature guarantees the performance of PSD by establishing the \emph{subgradient dominance property}, under which every stationary point is globally optimal (e.g., \citet{wang2023policy,kitamura2025near}). These results appear mathematically sound, as they sidestep classical NP-hardness results for RMDPs, such as those of \citet{wiesemann2013robust}, by assuming access to exact subgradients.
Unfortunately, we show that this line of analysis is flawed and formally refute the claimed PSD guarantees.\footnote{An erratum in \citet{kitamura2025near} acknowledged errors in the proofs of both \citet{kitamura2025near} and \citet{wang2023policy}, but did not provide a counterexample to subgradient dominance in general RMDPs.}
In particular, we prove that \textbf{the RMDP objective is not subgradient-dominant in general}. Even a very simple RMDP with a finite uncertainty set of cardinality two admits a suboptimal strict local minimum, and PSD can get stuck in its neighborhood (\cref{example:RMDP-failure}).

\looseness=-1
Moreover, we establish a stronger negative result: finding an $\varepsilon$-optimal policy can be NP-hard for {\rm (i)} finite transition uncertainty sets (\cref{proposition:NP-hardness}) and {\rm (ii)} $sa$-rectangular finite transition uncertainty sets with finite non-rectangular cost uncertainty sets (\cref{proposition:NP-hard-sa-rect}). 
Unlike prior NP-hardness results on objective evaluation \citep{wiesemann2013robust} or exact optimization \citep{ou2025sequential,bagnell2001solving}, our result shows that approximate policy optimization can remain intractable even when subgradients are efficiently computable, unless $\text{P}=\text{NP}$.

\looseness=-1
Despite these negative results, PSD can succeed under additional structure. We identify two sufficient conditions under which PSD finds $\varepsilon$-optimal policies (\cref{theorem:PSD-succeeds,corollary:PSD-convergence}): when, for each policy, either the worst-case transition kernel is unique or a robust action-value function exists.
These conditions cover several previously studied settings, including regularized RMDPs \citep{yang2023robust,zhang2024soft,derman2021twice}, $sa$- and $r$-rectangular uncertainty sets \citep{iyengar2005robust,nilim2005robust,goyal2023robust}, cost-robust MDPs \citep{Brekelmans2022,gadot2024solving}, and convex MDPs \citep{zhang2020variational,zahavy2021reward}.
A summary of our results is provided in \cref{table:results-summary}.
Overall, our findings clarify existing misunderstandings about PSD in RMDPs and provide precise conditions under which its guarantees can be restored.

\looseness=-1
The remainder of the paper is organized as follows. \cref{sec:problem-setup} introduces the RMDP formulation and the PSD algorithm. \cref{sec:failure of PSD} presents our negative results for finite uncertainty sets. \cref{sec:PSD-succeeds} establishes sufficient conditions for the success of PSD. Finally, \cref{sec:discussion} situates our contributions within the related literature and discusses the limitations of our work.

\begin{table}[t]
\small
\setlength{\tabcolsep}{3pt}
\centering
\begin{tabular}{|m{4.5cm}|m{3.4cm}|m{5.3cm}|}
\hline
\textbf{Setting} & \textbf{Subgrad. dom.?} & \textbf{Example} \\
\hline
\rowcolor{gray!15}
Finite $\cP$ $\times$ Singleton $\cC$ & No, NP-hard \newline {\scriptsize(\cref{proposition:NP-hardness})} & \cref{example:RMDP-failure}\\
\hline
\rowcolor{gray!15}
$sa$-rectangular finite $\cP$ $\times$ Finite non-rectangular $\cC$ & No$^\dagger$, NP-hard \newline {\scriptsize(\cref{proposition:NP-hard-sa-rect})} & Robust constrained MDPs with rectangular $\cP$ {\scriptsize \citep{kitamura2025near,marectified,chen2025robust}}\\
\hline
Worst-case transition kernel is unique for each policy {\scriptsize(\cref{assm:unique-worst-case-occupancy})} & Yes {\scriptsize(\cref{theorem:PSD-succeeds})}& MDPs {\scriptsize\citep{agarwal2021theory}}, \newline
Regularized RMDPs {\scriptsize(\cref{sec:PSD-succeeds-unique-model})}, \newline
Cost-robust MDPs {\scriptsize\citep{husain2021regularized}}, \newline
Convex MDPs {\scriptsize\citep{zahavy2021reward}}\\
\hline
Robust $Q$ function exists for each policy {\scriptsize(\cref{assm:unique-robust-value})} & Yes {\scriptsize(\cref{theorem:PSD-succeeds})} & 
MDPs {\scriptsize\citep{agarwal2021theory}}, \newline
Regularized RMDPs {\scriptsize\citep{yang2023robust,zhang2024soft}}, \newline
$sa$- and $r$-rectangular RMDPs {\scriptsize\citep{iyengar2005robust,goyal2023robust}} \\
\hline
$s$-rectangular RMDPs & Open problem & \\
\hline
\end{tabular}
\caption{\small Summary of the subgradient dominance property across different RMDP settings. Here, $\cC$ and $\cP$ denote the cost and transition uncertainty sets, respectively. Under \cref{assumption:init-dist}, PSD is guaranteed to find $\varepsilon$-optimal policies in settings marked ``Yes'', while those marked ``No'' are NP-hard even to approximate (\cref{proposition:NP-hardness,proposition:NP-hard-sa-rect}). $^\dagger$Unless $\text{P}=\text{NP}$, subgradient dominance cannot hold with a polynomially bounded constant (\cref{sec:NP-hardness-finite-uncertainty}). Highlighted settings were previously believed to be subgradient-dominant \citep{wang2023policy,li2023policy,wang2024policy,kitamura2025near}.}
\label{table:results-summary}
\end{table}

\input{sections/preliminary.tex}
\input{sections/failure-of-PSD.tex}
\input{sections/rectangular.tex}
\input{sections/discussion.tex}

\begin{ack}
AG acknowledges NSF EPCL Awards \#2534262 and \#2621763.
CS gratefully acknowledges funding from the Canada CIFAR AI Chairs Program, Amii, and NSERC.
\end{ack}

\bibliography{main}
\bibliographystyle{plainnat}

\appendix
\crefalias{section}{appendix}

\input{sections/useful-lemma.tex}

\end{document}

%% file: sections/preliminary.tex
\section{Preliminaries}\label{sec:problem-setup}

\paragraph{Basic notation.}
\looseness=-1
The probability simplex over a finite set $\S$ is denoted by $\Delta(\S)$.
For an integer $n$, let $[n]\df \brace*{1,\dots, n}$.
We define $\bzero \df \paren{0, \ldots, 0}^\top$ and $\bone \df \paren{1, \ldots, 1}^\top$.
For a set $\mathcal{X} \subset \R^d$, we define $\dist(\bzero; \mathcal{X}) \df \inf_{x \in \mathcal{X}} \norm{x}_2$.
For $\cX \subset \R^d$, $\conv \cX$ denotes its convex hull.
The indicator function $\delta_\cX$ of a set $\cX$ takes value $0$ at $x \in \cX$ and $+\infty$ otherwise.
All scalar operations and inequalities applied to vectors or functions are understood elementwise.

\looseness=-1
For a function $f : \R^d \to \R$, we let $\partial f(x)$ denote the set of Fr\'echet subgradients of $f$ at $x \in \R^d$ \citep[Definition~8.3]{rockafellar2009variational}.
If $\partial f(x)$ is a singleton, we denote its element by $\nabla f(x)$ and refer to it as the gradient of $f$ at $x$.
A point $x$ is said to be \emph{stationary} for $f$ if $\bzero \in \partial f(x)$.

\subsection{Robust MDPs}\label{subsec:RMDP}

\looseness=-1
An infinite-horizon tabular RMDP is defined by a tuple $(\S, \A, \mu, \gamma, \cU)$, where $\S$ and $\A$ are finite state and action spaces, respectively, $\mu \in \Delta(\S)$ is the initial state distribution, and $\gamma \in [0,1)$ is the discount factor.
The set $\cU$ is a compact uncertainty set of cost functions and transition kernels.
We refer to a pair $(c, P) \in \cU$ as a \emph{model}.
For each $(c, P) \in \cU$, the cost function $c : \S \times \A \to [0,1]$ specifies the cost of taking action $a$ in state $s$, and the transition kernel $P$ satisfies $P(\cdot \mid s,a) \in \Delta(\S)$.

\looseness=-1
When the uncertainty set decomposes as $\cU = \cC \times \cP$, we use $\cC$ and $\cP$ to denote the cost and transition sets, respectively.
If both $\cC$ and $\cP$ are singletons, the RMDP reduces to a standard MDP.
If $\cP$ is a singleton, the problem is known as a \emph{cost-robust MDP} \citep{husain2021regularized,gadot2024solving}, which encompasses a variety of decision-making problems (e.g., convex MDPs in \cref{sec:PSD-succeeds-general-cost}).

\looseness=-1
A (stationary Markov) policy $\pi$ is a probability kernel such that $\pi\paren{\cdot \given s} \in \Delta(\A)$ specifies the action distribution at state $s \in \S$. The set of all policies is denoted by $\Pi \subset \R^{\aS\times \aA}$, which corresponds to the \emph{direct parameterization} policy class \citep{agarwal2021theory}. 

\looseness=-1
Given a transition kernel $P$, the occupancy measure $\occ{\pi}_{P}: \SA \to \R$ represents the $\gamma$-discounted visitation frequency of each state--action pair under policy $\pi$:
$$
  \occ{\pi}_{P}(s, a) =
  (1-\gamma)
  \sum_{h=0}^\infty
  \gamma^{h}
  \P^\pi_P\paren*{s_h=s, a_h = a}\;,
$$
where the probability is taken over trajectories generated by
$s_0 \sim \mu$, $a_h \sim \pi \paren{\cdot \given s_h}$, and $s_{h+1} \sim P \paren{\cdot \given s_h, a_h}$.
The induced state occupancy measure is defined as $\socc{\pi}_{P}(s) \df \sum_{a \in \A} \occ{\pi}_{P}(s,a)$.

\looseness=-1
The value function $\vf{\pi}_{c, P} : \S \to \R$ (and the action-value function $\qf{\pi}_{c, P} : \SA \to \R$) represents the expected total cost under policy $\pi$ starting from a state $s$ (and from a state--action pair $(s,a)$, respectively). They are the unique solution to the following \emph{Bellman equations} \citep{puterman1990markov}:
\begin{equation*}
  \begin{aligned}
    \vf{\pi}_{c, P}(s) &= c^\pi (s) + \gamma \sum_{s' \in \S}P^\pi\paren{s'\given s}\vf{\pi}_{c, P}(s')\quad \forall s \in \S
    \\
    \text{and}\quad
    \qf{\pi}_{c, P}(s, a) &= c(s, a) + \gamma \sum_{s' \in \S}P\paren*{s'\given s, a}\vf{\pi}_{c, P}(s')\quad \forall (s, a) \in \SA\;,
  \end{aligned}
\end{equation*}
where we use shorthands $c^\pi(s) \df \sum_a \pi\paren{a\given s} c(s, a)$ and $P^\pi\paren{s'\given s} \df \sum_a P\paren{s'\given s, a} \pi\paren{a\given s}$. We define the advantage function as $A^\pi_{c, P}(s, a) \df \vf{\pi}_{c, P}(s) - \qf{\pi}_{c, P}(s, a)$.

\looseness=-1
The robust total cost of a policy $\pi$ is defined as
\begin{equation}\label{eq:robust-total-cost-def}
  J_{\cU}(\pi)= \max_{(c, P) \in \cU} \E^\pi_P\brack*{
    \sum_{h=0}^\infty
    \gamma^{h}
    c(s_h, a_h)}
  = \max_{(c, P) \in \cU} \sum_{s \in \S}\mu(s) \vf{\pi}_{c, P}(s)\;.
\end{equation}
If $\cU = \cC \times \cP$ and either $\cC$ or $\cP$ is a singleton, we write $\cU = \cC$ or $\cU = \cP$ for simplicity.
We denote $J_{c,P} \df J_{\{(c,P)\}}$ as the total cost under a fixed model and refer to $J_{\cU}$ as the robust total cost.

\looseness=-1
The goal of an RMDP is to identify an optimal policy $\pi^\star$ that minimizes the robust total cost:
\begin{align}\nonumber
  \text{(Robust MDP)}\quad\quad
  \pi^\star \in \argmin_{\pi \in \Pi} J_{\cU}(\pi)\;.
\end{align}
We call a policy $\pi$ \emph{$\varepsilon$-optimal} if it satisfies $J_{\cU}(\pi) - J_{\cU}(\pi^\star) \leq \varepsilon$.

\subsection{Rectangularity for Transition Uncertainty Sets}\label{sec:rectangularity}

\looseness=-1
Without any assumptions on the transition set $\cP$, computing the robust total cost $J_{\cP}(\pi)$ is NP-hard, even to approximate \citep{wiesemann2013robust}.
A widely adopted regularity condition is $sa$-rectangularity \citep{iyengar2005robust,nilim2005robust}:
$$
\text{$sa$-rectangularity} \quad 
\cP = \times_{s, a}\; \cP_{s, a} \quad \text{where} \quad \cP_{s, a} \subseteq \Delta(\S)\;.
$$
Here, $\times_{s,a}$ denotes the Cartesian product over all state--action pairs.
Intuitively, $sa$-rectangularity allows the adversary to choose worst-case transitions independently for each state--action pair.

\looseness=-1
Because this assumption can yield overly conservative policies, several works have proposed weaker notions of rectangularity, most notably \emph{$s$-rectangularity} \citep{wiesemann2013robust} and \emph{$r$-rectangularity} \citep{goyal2023robust}:
\begin{align}
\text{$s$-rectangular} \quad &\cP = \times_{s}\; \cP_{s} \quad \text{where} \quad \cP_{s} \subseteq \Delta(\S)^{\aA}\;,\nonumber\\
\text{$r$-rectangular} \quad &\cP = \brace*{
P\given P\paren{s'\given s, a} = \sum^r_{i=1} \phi_i(s, a) w_i(s')
}\;,\nonumber\\[-1mm]
&\text{where } (w_1, \dots, w_r) \in \mathcal{W}_1\times \dots \times \mathcal{W}_r\;. \nonumber
\end{align}
Here, $\phi(s,a) \in \Delta([r])$ and $\cW_i \subseteq \Delta(\S)$ for each $i \in [r]$.
Unlike $sa$-rectangularity, $s$-rectangular uncertainty sets allow the adversary to select worst-case transitions independently across states, but not across actions.
The $r$-rectangular structure resembles low-rank MDP models \citep{jin2020provably}, in which the transition kernel admits a linear decomposition via a feature map $\phi$ of dimension $r$.
Although both notions relax $sa$-rectangularity, they are not nested: neither $s$-rectangularity nor $r$-rectangularity contains the other \citep[Proposition~2.3]{goyal2023robust}.

\looseness=-1
These forms of rectangularity enable efficient computation of the robust total cost $J_{\cP}(\pi)$ using dynamic programming (DP) methods.
Moreover, an $\varepsilon$-optimal policy can be computed in time polynomial in $\aS$, $\aA$, $(1-\gamma)^{-1}$, $\log \varepsilon^{-1}$, and the description length of $\cP$ \citep{wiesemann2013robust,goyal2023robust}.

\subsection{Projected Subgradient Descent for Robust MDPs}\label{subsec:math-notation}

\looseness=-1
Beyond DP methods for rectangular uncertainty sets, an alternative line of work has explored \emph{projected subgradient descent} (PSD) for solving RMDPs \citep{wang2023policy,kitamura2025near,li2023policy,wang2024policy}. 
Unlike DP-based approaches, PSD does not rely on strong structural assumptions on the uncertainty sets $\cC$ or $\cP$. Instead, it requires access only to a worst-case policy subgradient.
This flexibility makes PSD applicable to classes of RMDPs that are not amenable to DP methods.
For instance, when the transition set is finite (e.g., $\cP = \{P_1, P_2\}$), PSD is applicable whereas DP methods generally are not.

\looseness=-1
At iteration $t \in \N$, PSD updates $\pi_t$ to a new policy $\pi_{t+1}$ by:
\begin{equation}\label{eq:pol-grad-update}
  \begin{aligned}
    \pi_{t+1} \df \proj_{\Pi}\paren*{\pi_t - \eta g_t}\quad \text{ where }\; g_t \in \partial J_{\cU} (\pi_t) \quad \forall t \in [T]\;,
  \end{aligned}
\end{equation}
\looseness=-1
where $\eta > 0$ is the learning rate and $\proj_{\Pi}$ denotes the Euclidean projection onto $\Pi$, which can be implemented efficiently \citep{duchi2008efficient}.
The subgradient $g_t$ can be efficiently computed when a worst-case model in \cref{eq:robust-total-cost-def} is available:

\begin{lemma}\label{lemma:policy subgradient}
  Let $\cU^\pi \df \argmax_{(c, P) \in \cU} J_{c, P}(\pi)$ denote the set of worst-case models under policy $\pi$. For any $\pi \in \Pi$, it holds that
  \begin{align}
    &\partial J_{\cU} (\pi) =
    \conv\brace*{
      \nabla J_{c, P}(\pi) \given (c, P) \in \cU^
      \pi
    }\label{eq:J-subgradient}
    \\
    \text{where}\quad &
    \paren*{\nabla J_{c, P}(\pi)}(s, a)=
    \frac{1}{1-\gamma} \socc{\pi}_{P}(s) \qf{\pi}_{c, P}(s, a) \quad
    \forall (s, a) \in \SA\;.\label{eq:J-gradient}
  \end{align}
\end{lemma}
\cref{eq:J-subgradient} follows from a version of Danskin's theorem (\cref{lemma:danskin's theorem} in \cref{appendix:useful-lemmas}) and \cref{eq:J-gradient} is known as the \emph{policy gradient theorem} \citep{xiao2022convergence}.

\looseness=-1
The convergence behavior of the PSD update in \cref{eq:pol-grad-update} can be analyzed using the \emph{Moreau envelope} of \emph{weakly convex} functions.

\begin{definition}[Moreau envelope]\label{def:moreau-envelope}
  For a function $f: \R^d \to \R$ and a parameter $\nu > 0$, the Moreau envelope is defined by $\ME_\nu \comp f: \R^d \to \R$ such that
  $$
    \ME_\nu \comp f(x) = \min_{y \in \R^d} \brace*{f(y) + \frac{1}{2\nu} \norm*{x - y}^2_2}\;.
  $$
  The minimal point is called the proximal point: $\prox_{\nu f}(x) = \argmin_{y \in \R^d} \brace*{f(y) + \frac{1}{2\nu} \norm*{x - y}^2_2}$.
\end{definition}

\begin{definition}[Weak convexity {\citep[Definition~2.1]{atenas2023unified}}]\label{def:weak-convex}
  \looseness=-1
  For $\omega \geq 0$, a function $f: \R^d \to \R$ is called $\omega$-weakly convex if $f(\cdot) + \frac{\omega}{2}\norm*{\cdot}_2^2$ is convex.
\end{definition}
\looseness=-1
Weak convexity is a mild regularity condition satisfied by many functions in optimization.
In particular, the pointwise maximum of smooth functions is weakly convex \citep[Proposition~2.4]{atenas2023unified}, and hence the robust total cost $J_{\cU}$ is weakly convex.
The Moreau envelope provides a smooth approximation of a weakly convex function.
Specifically, if $f$ is proper and $\omega$-weakly convex, then for any $\nu \in (0,1/\omega)$, the Moreau envelope $\ME_\nu \comp f$ is differentiable, with gradient given by \citep[Lemma~4.3]{drusvyatskiy2019efficiency}:
\begin{equation}\label{eq:Moreau-grad}
\nabla \ME_\nu \comp f(x) = \frac{1}{\nu} \paren*{x - \prox_{\nu f}(x)} \quad \forall x \in \R^d\;.
\end{equation}

\looseness=-1
Combining \cref{eq:Moreau-grad} with the definition of the Moreau envelope yields
\begin{equation}\label{eq:Moreau-grad-to-subgrad}
\dist(\bzero; \partial f(\prox_{\nu f}(x))) \leq \norm*{\nabla \ME_\nu \comp f(x)}_2
= \frac{1}{\nu} \norm*{x - \prox_{\nu f}(x)}_2\;.
\end{equation}
Thus, a small Moreau-envelope gradient implies that $x$ is close to a point $\prox_{\nu f}(x)$ that is nearly stationary for $f$. 
For further discussion of weak convexity and the Moreau envelope, see
\citep{davis2019stochastic,renaud2025moreau,rockafellar2009variational}.

\looseness=-1
Applying these tools to RMDPs yields convergence guarantees for PSD.
Define the extended objective $\barJ_{\cU}: \pi \in \R^{\aS\times \aA}\mapsto J_{\cU}(\pi) + \delta_{\Pi}(\pi)$ which extends the domain of the robust total cost to $\R^{\aS\times \aA}$.
Let $\rho \df \frac{2\aA}{(1-\gamma)^3}$. Since each total cost $J_{c,P}$ is $\frac{2\gamma\aA}{(1-\gamma)^3}$-smooth (\cref{lemma:MDP-smooth}) and $\gamma \leq 1$, the extended objective $\barJ_{\cU}$ is $\rho$-weakly convex.
Standard analyses of PSD for weakly convex functions imply that the iterates produced by \cref{eq:pol-grad-update} converge to near-stationary points of the Moreau envelope (see \cref{sec:subgrad-dom-Moreau-derivation}).

\begin{lemma}\label{lemma:stationary-convergence}
    When $\eta = 1/\sqrt{T}$, the PSD iterates in \cref{eq:pol-grad-update} satisfy
    \begin{align}
    \min_{t \in [T]} \norm*{\nabla (\ME_{\frac{1}{2\rho}} \comp \barJ_{\cU})(\pi_t)}_2
    \leq T^{-1/4} \sqrt{\frac{4}{1-\gamma} + \frac{4 \aA^2}{(1-\gamma)^7}}\;.\nonumber
    \end{align}
\end{lemma}
Together with \cref{eq:Moreau-grad-to-subgrad}, \cref{lemma:stationary-convergence} implies that there exists an iterate $\pi_t$ that lies close to a proximal point which is nearly stationary for $\barJ_{\cU}$.

%% file: sections/failure-of-PSD.tex
\section{Failure of Projected Subgradient Descent in Robust MDPs}\label{sec:failure of PSD}

\looseness=-1
Because the total cost function is nonconvex in the policy parameterization \citep{agarwal2021theory}, the stationarity guarantee in \cref{lemma:stationary-convergence} alone does not imply near-optimality.
To bridge this gap, prior works have attempted to establish the \emph{subgradient dominance} property for the robust total cost $J_{\cU}$ \citep{wang2023policy,kitamura2025near}:\footnote{\citet{wang2023policy} adopt the formulation in \cref{eq:subgrad-dom-Moreau}, while \citet{kitamura2025near} use the stronger form in \cref{eq:strong-subgrad-dom}. Both notions are sufficient to guarantee that PSD efficiently finds an $\varepsilon$-optimal policy.}
\begin{definition}[Subgradient dominance]\label{def:subgrad-dom}
The robust total cost $J_{\cU}$ is said to be \emph{subgradient-dominant} with constant $D>0$ if, for every $\pi \in \Pi$,
\begin{equation}\label{eq:strong-subgrad-dom}
J_{\cU}(\pi) - J_{\cU}(\pi^\star)
\le D G(\pi) \; \text{ where } \;
G(\pi) \df \min_{g \in \partial J_{\cU}(\pi)}\max_{\pi'\in \Pi} \sum_{s, a} \paren*{\pi(a|s) - \pi'(a|s)} g(s, a) \;.
\end{equation}
\end{definition}
In the standard MDP setting where $\cU$ is a singleton, this condition reduces to the classical \emph{gradient dominance} property, which is known to hold.
\begin{assumption}\label{assumption:init-dist}
    The initial distribution $\mu$ has full support, i.e., $\mu(s) > 0$ for all $s \in \cS$.
\end{assumption}
\begin{lemma}[gradient dominance \citep{agarwal2021theory}]\label{lemma:gradient dominance}
  When $\cU$ is a singleton and \cref{assumption:init-dist} holds\footnote{Such a coverage condition is necessary to ensure the global convergence of policy gradient methods \citep{mei2020global}.}, the total cost function $J_{\cU}$ satisfies \cref{eq:strong-subgrad-dom} with constant $D = ((1-\gamma)\min_s \mu(s))^{-1}$.
\end{lemma}

\looseness=-1
Intuitively, the quantity $G(\pi)$, which is the smallest Frank--Wolfe gap among the subgradients at $\pi$, measures the first-order stationarity of $\pi$ with respect to the robust objective.
If $J_{\cU}$ is subgradient-dominant, the Lipschitz continuity of $J_{\cU}$ relates $G(\pi)$ to the gradient norm of the Moreau envelope and yields the following inequality (see \cref{sec:subgrad-dom-Moreau-derivation} for details):
\begin{equation}\label{eq:subgrad-dom-Moreau}
J_{\cU}(\pi) - \min_{\pi'\in\Pi} J_{\cU}(\pi') \le D' \norm*{\nabla (\ME_{\frac{1}{2\rho}} \comp \barJ_{\cU})(\pi)}_2\; \text{ where }\; D' \df 2D\sqrt{\aS} + \frac{1-\gamma}{4\sqrt{\aA}}\;.
\end{equation}
Consequently, under subgradient dominance, combining \cref{lemma:stationary-convergence} with \cref{eq:subgrad-dom-Moreau} guarantees that PSD finds an $\varepsilon$-optimal policy within $O(\varepsilon^{-4})$ policy updates.

\subsection{A Counterexample Where PSD Fails}\label{sec:counterexample-PSD-fails}

\looseness=-1
Consider RMDPs with uncertainty only in the transition kernel, i.e., $\cU = \cP$.
Several prior works claim that, under the full-support assumption on the initial distribution (\cref{assumption:init-dist}), the robust total cost $J_{\cP}$ is subgradient-dominant for any compact uncertainty set $\cP$ (e.g., Theorem 3.2 in \citet{wang2023policy} and Theorem 4 in \citet{kitamura2025near}).
We show that these claims are incorrect.
\textbf{In fact, the robust total cost need not be subgradient-dominant even for a simple transition set.}
Specifically, when $\abs{\cP} = 2$, the function $J_{\cP}$ can already admit a suboptimal strict local minimum.
The technical errors in the prior analyses are explained in \cref{sec:proof-error-in-prior-work}.

\begin{figure}
    \centering
    \begin{subfigure}[\small A deterministic RMDP with $\cP=\brace{P_1, P_2}$. Transitioning from state $s_+$ incurs a cost of $+1$. 
    ]
    {
      \includegraphics[width=0.6\linewidth]{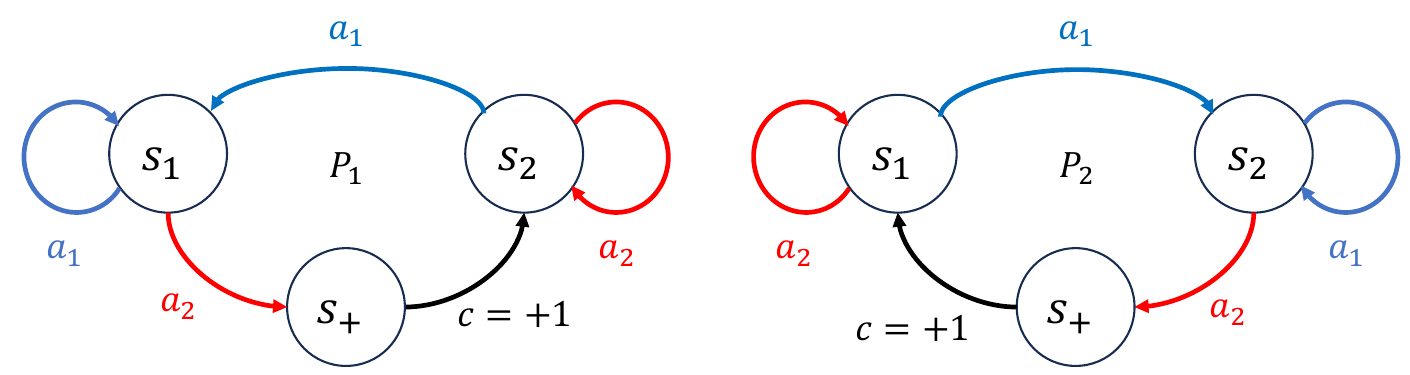}
    }
    \end{subfigure}
    \hfill
    \begin{subfigure}[\small Robust total cost landscape with $\gamma = 0.9$.]{
      \includegraphics[width=0.33\linewidth]{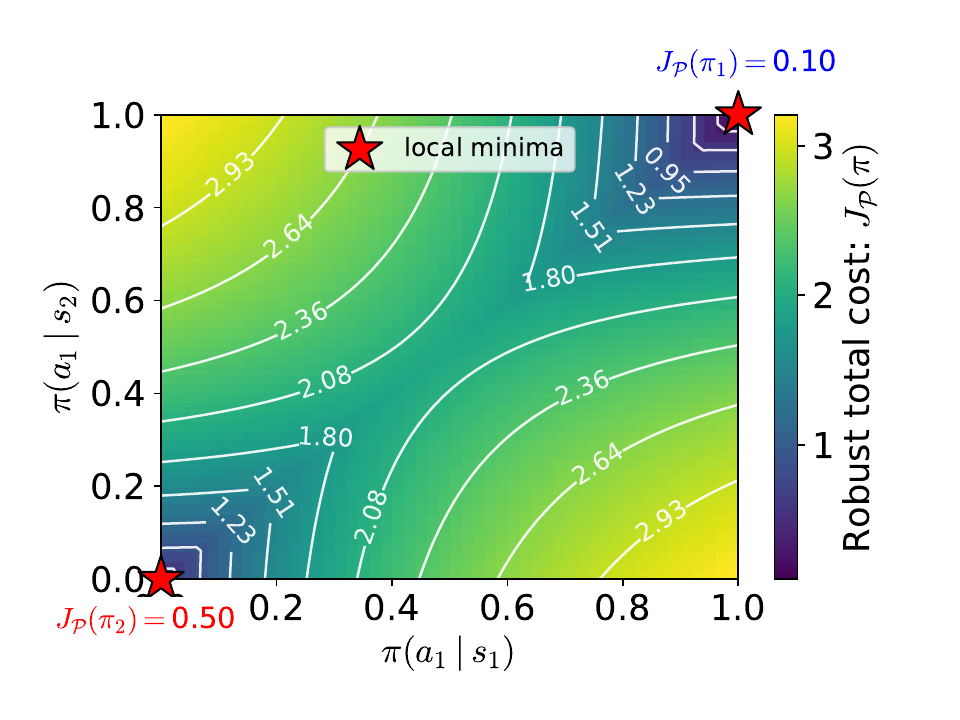}
    }
    \end{subfigure}
    \caption{\small 
    An RMDP where PSD fails (\textbf{left}) and its robust total cost landscape (\textbf{right}). The policy ${\color{red}\pi_2}$, which always selects action ${\color{red}a_2}$, is a strictly suboptimal local minimum. Details are provided in \cref{example:RMDP-failure}.}
  \label{fig:nonrect}
\end{figure}

\begin{example}[An RMDP where PSD fails]\leavevmode\phantomsection\label{example:RMDP-failure}%
  \looseness=-1
  Consider a deterministic RMDP with three states $(s_1, s_2, s_+)$, two actions $({\color{blue}a_1}, {\color{red}a_2})$, and two transition kernels $\cP = \{P_1, P_2\}$, as illustrated in \cref{fig:nonrect} (\textbf{left}). 
  We set the discount factor to $\gamma = 0.9$ and the initial distribution to $\mu(s_1) = \mu(s_2)=0.45$, which satisfies \cref{assumption:init-dist}. 
  Let ${\color{blue}\tilde{\pi}_1}$ and ${\color{red}\tilde{\pi}_2}$ denote the policies that always choose ${\color{blue}a_1}$ and ${\color{red}a_2}$, respectively. The policy ${\color{red}\tilde{\pi}_2}$ visits the costly state $s_+$ and therefore incurs a larger robust total cost than ${\color{blue}\tilde{\pi}_1}$.
  For this instance, the following proposition holds. The proof is given in \cref{sec:proof-of-psd-trap}.  
\begin{proposition}\label{proposition:psd-trapped-near-pi2}
  The following two claims hold in the example instance:
    \begin{itemize}
      \item ${\color{red}\tilde{\pi}_2}$ is a suboptimal strict local minimum of $J_{\cP}$ satisfying $0.505=J_{\cP}({\color{red}\tilde{\pi}_2}) > J_{\cP}({\color{blue}\tilde{\pi}_1})=0.1$.
      \item For all sufficiently small $\eta>0$, when initialized at $\pi_1 = {\color{red}\tilde{\pi}_2}$, the iterates of the PSD update
      $\pi_{t+1} = \proj_{\Pi}\paren*{\pi_t - \eta g_t}$ with any $g_t \in \partial J_{\cP}(\pi_t)$ satisfy $\|\pi_t - {\color{red}\tilde{\pi}_2}\|_\infty \leq 3\eta$ for all $t\ge 1$.
    \end{itemize}
\end{proposition}
  This result formally shows that $J_{\cP}$ is not subgradient-dominant: since ${\color{red}\tilde{\pi}_2}$ is a local minimum, the first-order optimality condition gives $G({\color{red}\tilde{\pi}_2}) = 0$, whereas $J_{\cP}({\color{red}\tilde{\pi}_2}) - J_{\cP}(\pi^\star) \geq 0.405$, violating \cref{eq:strong-subgrad-dom} for every $D > 0$.
  Moreover, PSD initialized at ${\color{red}\tilde{\pi}_2}$ remains in an $O(\eta)$-neighborhood of this suboptimal policy. These facts are also illustrated empirically in \cref{fig:nonrect} (\textbf{right}), which shows the landscape of the robust total cost function.
\end{example}

\subsection{\texorpdfstring{NP-hardness of $\varepsilon$-Optimal Policy Identification}{NP-hardness of ε-Optimal Policy Identification}}\label{sec:NP-hardness-finite-uncertainty}

\looseness=-1
Since PSD fails to find an $\varepsilon$-optimal policy under general uncertainty sets, a natural question is whether any algorithm can efficiently identify an $\varepsilon$-optimal policy in such settings. Unfortunately, the answer is negative. 
Even when the transition set is finite, identifying an $\varepsilon$-optimal policy is NP-hard.
\begin{proposition}[Hardness for finite $\cP$ and singleton $\cC$]\label{proposition:NP-hardness}
  For any fixed $\gamma\in(0,1)$ and any $\varepsilon < 0.5 \gamma^3 (1-\gamma)^{-1}$, finding an $\varepsilon$-optimal policy in RMDPs with a finite transition set and a singleton cost set is NP-hard.
\end{proposition}
\looseness=-1
The proof is based on a reduction from the \textsc{3-SAT} problem (see \cref{sec:proof-of-NP-hard}). 
Unlike prior NP-hardness results on objective evaluation \citep{wiesemann2013robust} or exact optimization \citep{ou2025sequential,bagnell2001solving}, this result shows that approximate policy optimization can remain NP-hard even when policies can be evaluated efficiently.
Indeed, in this setting, since the uncertainty set is finite, both the robust total cost and its subgradients can be approximated to $\delta$ accuracy in time polynomial in $\aS$, $\aA$, $(1-\gamma)^{-1}$, $\log \delta^{-1}$, and $\abs{\cU}$ \citep{kumar2024policy}.
A more detailed comparison with prior NP-hardness results is provided in \cref{sec:discussion}.

\looseness=-1
As a byproduct, we obtain an additional hardness result for RMDPs with a rectangular transition set and a finite cost set.
This setting is especially relevant for \emph{robust constrained MDPs} (RCMDPs), in which a policy must satisfy multiple robust constraints \citep{kitamura2025near,russel2020robust}:
\begin{align*}
   \text{(RCMDP)} \quad \min_{\pi \in \Pi} J_{c_0, \cP}(\pi) \quad \text{such that} \quad \max_{n \in [N]} J_{c_n, \cP}(\pi) \leq b\;. 
\end{align*}
Here, $c_0, \ldots, c_N$ are the objective and constraint cost functions, and $b > 0$ is a constraint threshold. 
When $\cP$ is rectangular, the constraint term $\max_{n \in [N]} J_{c_n,\cP}(\pi)$ corresponds to an RMDP with a rectangular transition set and a non-rectangular (finite) cost set $\cC=\{c_1,\ldots,c_N\}$.

\looseness=-1
Recent work by \citet{kitamura2025near} proposes a PSD-based algorithm for RMDPs with a general transition set and a finite cost set, suggesting an efficient approach for solving RCMDPs.
However, the following \cref{proposition:NP-hard-sa-rect} shows that finding an $\varepsilon$-optimal policy in an RCMDP is NP-hard, even when the transition set is $sa$-rectangular and finite.

\begin{proposition}[Hardness for $sa$-rectangular $\cP$ and finite non-rectangular $\cC$]\label{proposition:NP-hard-sa-rect}
\looseness=-1
For any fixed $\gamma\in(0,1)$ and any $\varepsilon < \frac{\gamma^2}{2(2-\gamma)}$, finding an $\varepsilon$-optimal policy in RMDPs with an $sa$-rectangular finite transition set and a finite non-rectangular cost set is NP-hard. Moreover, determining the feasibility of RCMDPs with $sa$-rectangular $\cP$ is NP-hard.
\end{proposition}
The proof constructs an RMDP with an $sa$-rectangular transition set and a finite cost set that simulates the RMDP instance in \cref{proposition:NP-hardness}.
Full details are provided in \cref{sec:proof-of-NP-hard}.

\looseness=-1
In the settings of \cref{proposition:NP-hardness,proposition:NP-hard-sa-rect}, the uncertainty set $\cU$ is finite.
Thus, for rational $\gamma$ and model parameters, the robust total cost and its subgradients can be computed exactly in polynomial time.
Unless $\text{P}=\text{NP}$, PSD therefore cannot find an $\varepsilon$-optimal policy within polynomially many updates, contrary to the claims of \citet{wang2023policy} and \citet{kitamura2025near}.\footnote{The constructions use $\mu(s_{\text{ini}}) = 1$, which violates \cref{assumption:init-dist}. However, since $J_\cU$ is $(1-\gamma)^{-1}$-Lipschitz in $\mu$, mixing $\mu$ with the uniform distribution with a small weight depending only on $\gamma$ preserves the hardness for a smaller $\varepsilon$.}
In particular, by \cref{lemma:stationary-convergence} and \cref{eq:subgrad-dom-Moreau}, $J_{\cU}$ cannot be subgradient-dominant with a polynomially bounded $D$.

%% file: sections/rectangular.tex
\section{Sufficient Conditions When Projected Subgradient Descent Succeeds}\label{sec:PSD-succeeds}

\looseness=-1
Although PSD can fail for general RMDPs, it succeeds under additional conditions.
In this section, we identify two sufficient conditions under which the subgradient dominance property \eqref{eq:strong-subgrad-dom} holds.

\begin{assumption}\label{assm:unique-worst-case-occupancy}
    Let $\cP^\pi \df \brace{P \given (c, P) \in \cU^\pi}$ be the set of worst-case transition kernels under policy $\pi$.
    The set $\cP^\pi$ is a singleton for any $\pi \in \Pi$.
\end{assumption}
\begin{assumption}\label{assm:unique-robust-value}
    For any $\pi \in \Pi$, the set $\cQ^\pi \df \brace{Q^\pi_{c, P} \given (c, P) \in \cU^\pi}$ of worst-case action-value functions contains a \emph{robust action-value function} $Q^\pi_\cU$, that is, $Q \leq Q^\pi_\cU$ for all $Q \in \cQ^\pi$.
\end{assumption}

\looseness=-1
Our assumptions are motivated by the failure mode of PSD in \cref{example:RMDP-failure}. Intuitively, suboptimal local minima can arise when different worst-case models prescribe conflicting local directions toward the global optimum. Our assumptions rule out such conflict. \Cref{assm:unique-robust-value} (e.g., a singleton $\cQ^\pi$) ensures that all worst-case models agree on a local improvement direction, since any action that improves on $\pi$ under $Q^\pi_\cU$ also improves on $\pi$ under every worst-case model. \Cref{assm:unique-worst-case-occupancy} requires agreement only at the transition level, since conflict in the cost is known not to create this pathology \citep{barakat2025global,fatkhullin2025stochastic}. The following theorem shows that either condition is sufficient for subgradient dominance.

\begin{theorem} [Sufficient conditions for subgradient dominance]\label{theorem:PSD-succeeds}
    \looseness=-1
    Under the full-support initial distribution (\cref{assumption:init-dist}), the robust total cost $J_\cU$ is subgradient-dominant with constant $D = ((1-\gamma)\min_s \mu(s))^{-1}$ if either \cref{assm:unique-worst-case-occupancy} or \cref{assm:unique-robust-value} holds.
\end{theorem}
\begin{proof}
    \looseness=-1
    Let $g^\star\in\partial J_{\cU}(\pi)$ attain the minimum in $G(\pi)$. By \cref{lemma:policy subgradient} and Carath\'eodory's theorem, there exist active models $(c_1,P_1),\ldots,(c_k,P_k)\in\cU^\pi$ with finite $k \in \N$, and $\alpha^\star\in\Delta([k])$ such that $g^\star=\sum_{i=1}^k\alpha_i^\star\nabla J_{c_i,P_i}(\pi)$. Thus, restricting the subgradient minimization in $G(\pi)$ to the convex hull of these gradients $\nabla J_{c_i,P_i}(\pi)$ leaves $G(\pi)$ unchanged. For these active models, it holds that
    \begin{align}
        J_{\cU}(\pi) - J_{\cU}(\pi^\star) 
        &\numeq{\leq}{a} \min_{i \in [k]} J_{c_i, P_i}(\pi) - J_{c_i, P_i}(\pi^\star) 
        \numeq{=}{b} \min_{i \in [k]} \frac{1}{1-\gamma} \sum_{s} \socc{\pi^\star}_{P_i}(s) \sum_a \pi^\star(a|s) A^{\pi}_{c_i, P_i}(s, a)\nonumber\\
        &\numeq{=}{c} \min_{\alpha \in \Delta([k])} \frac{1}{1-\gamma} \sum_{i=1}^k \alpha_i \sum_{s} \socc{\pi^\star}_{P_i}(s)\sum_a \pi^\star(a|s) A^{\pi}_{c_i, P_i}(s, a) \fd \circled{1}\;,\label{eq:alpha-min-expression}
    \end{align}
    \looseness=-1
    where (a) uses $J_{\cU}(\pi^\star) \geq J_{c, P}(\pi^\star)$ for any $(c, P)$, (b) uses the performance difference lemma for MDPs (\cref{lemma:performance-difference}), and (c) holds since the minimum is attained at an extreme point of the simplex.

    \paragraph{Under \cref{assm:unique-worst-case-occupancy}.} Let $P \in \cP^\pi$ be the unique worst-case transition kernel under $\pi$. The right-hand side of \cref{eq:alpha-min-expression} is bounded as
    \begin{align*}
        \circled{1}
        &\leq \frac{1}{1-\gamma} \min_{\alpha \in \Delta([k])} \sum_{s} \socc{\pi^\star}_P(s) \underbrace{\max_a \sum_{i=1}^k \alpha_i A^{\pi}_{c_i, P}(s, a)}_{\geq 0}\\
        &\numeq{\leq}{a} \frac{1}{1-\gamma} \min_{\alpha \in \Delta([k])} \sum_{s} \frac{1}{(1-\gamma)\mu(s)} \socc{\pi}_P(s)\max_a \sum_{i=1}^k \alpha_i A^{\pi}_{c_i, P}(s, a)\\
        &\numeq{\leq}{b} D \min_{\alpha \in \Delta([k])}\max_{\pi'\in \Pi} \sum_{s,a} \paren*{\pi(a|s) - \pi'(a|s)}\frac{1}{1-\gamma} \socc{\pi}_P(s)\sum_{i=1}^k \alpha_i \qf{\pi}_{c_i, P}(s, a)
        \numeq{=}{c} D G(\pi)\;,
    \end{align*}
    where (a) uses $\socc{\pi^\star}_P(s) / \socc{\pi}_P(s) \leq 1 / ((1-\gamma)\mu(s))$ and (b) uses $A^\pi_{c, P}(s, a) = \sum_{a'} \pi(a'|s) Q^\pi_{c, P}(s, a') - Q^\pi_{c, P}(s, a)$ and introduces $D \df ((1-\gamma)\min_s \mu(s))^{-1}$. The equality (c) follows from the choice of active models above. This proves \cref{theorem:PSD-succeeds} under \cref{assm:unique-worst-case-occupancy}.

    \paragraph{Under \cref{assm:unique-robust-value}.}
    \looseness=-1
    Choose $(c_1, P_1)$ above to attain $Q^\pi_\cU$ (with $\alpha^\star_1 = 0$ if necessary), and let $A^\pi_\cU$ be its advantage function.
    Since $Q^\pi_{c_i, P_i} \leq Q^\pi_\cU$ and $J_{c_i, P_i}(\pi) = J_{c_1, P_1}(\pi)$, \cref{assumption:init-dist} gives $V^\pi_{c_i, P_i} = V^\pi_{c_1, P_1}$, so $A^\pi_{c_i, P_i} \geq A^\pi_\cU$ and
    \begin{align*}
        \circled{1}
        &\numeq{\leq}{a} \frac{1}{1-\gamma} \sum_{s} \socc{\pi^\star}_{P_1}(s) \underbrace{\max_a A^{\pi}_\cU(s, a)}_{\geq 0}\\
        &\numeq{\leq}{b} \frac{1}{1-\gamma} \min_{\alpha} \sum_{s}\sum_{i=1}^k \alpha_i \frac{1}{(1-\gamma)\mu(s)} \socc{\pi}_{P_i}(s)\max_a A^{\pi}_{\cU}(s, a)\\
        &\numeq{\leq}{c} D \min_{\alpha}\max_{\pi'\in \Pi} \sum_{s,a} \paren*{\pi(a|s) - \pi'(a|s)}\sum_{i=1}^k \alpha_i \frac{1}{1-\gamma} \socc{\pi}_{P_i}(s) \qf{\pi}_{c_i, P_i}(s, a)
        \numeq{=}{d} D G(\pi)\;,
    \end{align*}
    where (a) takes $\alpha$ to be the first unit vector in $\circled{1}$, (b) uses $\socc{\pi^\star}_{P_1}(s) \leq 1 \leq \sum_{i} \alpha_i \socc{\pi}_{P_i}(s) / ((1-\gamma)\mu(s))$, (c) uses $A^\pi_\cU \leq A^\pi_{c_i, P_i}$ as in the previous case, and (d) follows from the choice of active models.
\end{proof}

\looseness=-1
Combining \cref{theorem:PSD-succeeds} with the stationarity guarantee in \cref{lemma:stationary-convergence} and the Moreau-envelope bound \cref{eq:subgrad-dom-Moreau}, we obtain the following convergence guarantee.

\begin{corollary}[PSD convergence in RMDPs]\label{corollary:PSD-convergence}
    Suppose that \cref{assumption:init-dist} holds and either \cref{assm:unique-worst-case-occupancy} or \cref{assm:unique-robust-value} holds.
    When $\eta = 1/\sqrt{T}$, the PSD iterates in \cref{eq:pol-grad-update} satisfy
    \begin{align}
    \min_{t \in [T]} J_{\cU}(\pi_t) - J_{\cU}(\pi^\star)
    \leq C T^{-1/4} \;,
    \quad \text{ where }\; C \df D'\sqrt{\frac{4}{1-\gamma} + \frac{4 \aA^2}{(1-\gamma)^7}}\;.\nonumber
    \end{align}
\end{corollary}
Here, $D'$ is defined in \cref{eq:subgrad-dom-Moreau}.
Using the above result, an $\varepsilon$-optimal policy is obtained after $\cO(\aA^4\aS^2 \varepsilon^{-4}(\min_s\mu(s))^{-4}(1-\gamma)^{-18})$ iterations.
This recovers the $\cO(\varepsilon^{-4})$ guarantee claimed by prior work \citep{wang2023policy,kitamura2025near}, although it is slower than the $\cO(\varepsilon^{-1})$ rate for standard MDPs \citep{xiao2022convergence}. Whether this rate is tight remains an open question.

\looseness=-1
In the following subsections, we discuss concrete settings in which each assumption holds.
Clearly, the standard MDP setting with a singleton uncertainty set $\cU = \{(c,P)\}$ satisfies both assumptions.
Nonetheless, the two assumptions are independent and apply to different classes of RMDPs.
A formal proof is provided in \cref{sec:proof of independence}.
\begin{proposition}\label{proposition:independent-conditions}
    Neither of \cref{assm:unique-worst-case-occupancy,assm:unique-robust-value} implies the other.
\end{proposition}

\subsection{Regularized RMDPs \texorpdfstring{(\cref{assm:unique-worst-case-occupancy,assm:unique-robust-value})}{(Assumptions \getrefnumber{assm:unique-worst-case-occupancy} and \getrefnumber{assm:unique-robust-value})}}\label{sec:PSD-succeeds-unique-model}

\looseness=-1
Both \cref{assm:unique-worst-case-occupancy,assm:unique-robust-value} are trivially satisfied when the worst-case model is unique for every policy $\pi$, i.e., when $\abs{\cU^\pi} = 1$. A common approach to enforcing $\abs{\cU^\pi}=1$ is to introduce a concave regularizer that renders the adversary's maximization objective strongly concave, thereby guaranteeing a unique worst-case model. For example, consider the following variant of KL-regularized RMDPs \citep[e.g.,][]{yang2023robust}:
\begin{align*}
    \min_{\pi \in \Pi}\max_{P \in \cP} \E^\pi_P\brack*{
        \sum_{h=0}^\infty
        \gamma^{h}
        \paren*{c(s_h, a_h) - \tau \sum_{a \in \A} D_{\mathrm{KL}}(P(\cdot \mid s_h, a) \| P_0(\cdot \mid s_h, a))}}\;,
\end{align*}
where $P_0$ is a nominal transition kernel and $\tau > 0$ is a regularization parameter. When $\cP$ is $sa$-rectangular with $\cP_{s,a}=\Delta(\S)$, the worst-case kernel is unique and satisfies
\begin{align*}
    \bar{P}(\cdot|s, a)
    = \argmax_{p \in \Delta(\S)}
    \brace*{\pi(a|s)\gamma p^\top V^\pi_{\cU} - \tau D_{\mathrm{KL}}(p \| P_0(\cdot|s, a))}
    \propto P_0(\cdot|s, a) \exp\paren*{\pi(a|s)\gamma V^\pi_{\cU} / \tau}\;,
\end{align*}
where $V^\pi_{\cU}$ denotes the value function of the worst-case model.

\looseness=-1
While most existing results focus on rectangular uncertainty sets \citep{yang2023robust,zhang2024soft,derman2021twice}, the assumption $\abs{\cU^\pi}=1$ itself does not require rectangularity.
As an example, consider the following $\ell_2$ regularization on $P$:
\begin{align*}
    \max_{P \in \cP} J_{P}(\pi) - \frac{\tau}{2} \norm{P - P_0}_F^2\;,
\end{align*}
where $\norm{\cdot}_F$ denotes the Frobenius norm.
Assuming that $\cP$ is convex, the total cost $J_P(\pi)$ is weakly concave in $P$ due to its smoothness \citep{atenas2023unified}, so the regularized objective is strongly concave for sufficiently large $\tau$. Then, the worst-case model is unique for every policy even when $\cP$ is non-rectangular, and both \cref{assm:unique-worst-case-occupancy,assm:unique-robust-value} hold. Characterizing broader classes of non-rectangular uncertainty sets that satisfy $\abs{\cU^\pi}=1$ is an interesting direction for future work.

\subsection{Cost-robust MDPs and Convex MDPs (\texorpdfstring{\cref{assm:unique-worst-case-occupancy}}{Assumption \getrefnumber{assm:unique-worst-case-occupancy}})}\label{sec:PSD-succeeds-general-cost}

\looseness=-1
\cref{assm:unique-worst-case-occupancy} is also satisfied in cost-robust MDPs \citep{husain2021regularized,gadot2024solving}, where uncertainty lies solely in the cost function, that is, when $\cU = \cC$ for a compact cost set $\cC$.

\looseness=-1
Cost-robust MDPs are closely related to \emph{convex MDPs} \citep{zhang2020variational,zahavy2021reward}, which study optimization problems that are convex in the occupancy measure under a fixed $P$:
\begin{align*}
    \min_{\pi \in \Pi} f(\occ{\pi}_P) \quad \text{where $f$ is convex and finite on an open set containing $\cD \df \brace*{\occ{\pi}_P \given \pi \in \Pi}$}\;.
\end{align*}
Convex MDPs generalize diverse decision-making problems, including skill learning \citep{eysenbach2019diversity}, inverse reinforcement learning \citep{belogolovsky2021inverse}, imitation learning \citep{ho2016generative}, pure exploration \citep{hazan2019provably}, and constrained MDPs \citep{altman1999constrained}. We refer readers to \citet{zahavy2021reward} for a comprehensive overview.

\looseness=-1
Clearly, convex MDPs generalize cost-robust MDPs via
$f(d) = \max_{c \in \cC} \sum_{s, a} d(s, a) c(s, a)$. Conversely, cost-robust MDPs include convex MDPs.
Let $f^*$ denote the convex conjugate of $f$ \citep[Chapter~11]{rockafellar2009variational}.
Then, for every $\pi \in \Pi$,
\begin{align*}
    f\paren*{\occ{\pi}_{P}}
    \numeq{=}{a} \max_{g \in \partial f(\cD)} \sum_{s,a}g(s,a)\occ{\pi}_{P}(s,a) - f^*(g)
    \numeq{=}{b} (1-\gamma) J_{\cC}(\pi)\;,
\end{align*}
where (a) uses the Fenchel--Young inequality, which is tight for $g \in \partial f(\occ{\pi}_P)$, and (b) uses $\sum_{s,a} \occ{\pi}_P(s,a) = 1$ with the bounded cost set $\cC \df \brace*{g - f^*(g) \given g \in \partial f(\cD)}$.
This equivalence shows that convex MDPs satisfy \cref{assm:unique-worst-case-occupancy} and thus enjoy the subgradient dominance property.

\looseness=-1
Unless $\text{P}=\text{NP}$, adding transition uncertainty breaks this property even for $sa$-rectangular $\cP$, since $f(d) = \max_{m} \sum_{s,a} d(s,a) c_m(s,a)$ recovers the NP-hard instances in \cref{proposition:NP-hard-sa-rect}.

\subsection{\texorpdfstring{$sa$}{sa}- and \texorpdfstring{$r$}{r}-Rectangular RMDPs (\texorpdfstring{\cref{assm:unique-robust-value}}{Assumption \getrefnumber{assm:unique-robust-value}})}\label{sec:PSD-succeeds-rectangular}

\looseness=-1
\Cref{assm:unique-robust-value} is known to hold for $sa$-rectangular RMDPs \citep{iyengar2005robust,nilim2005robust}.
We show that it extends to the strictly more general class of $r$-rectangular RMDPs.

\begin{proposition}\label{proposition:r-rect-unique-value}
    Under \cref{assumption:init-dist}, \cref{assm:unique-robust-value} holds for $r$-rectangular $\cP$.
\end{proposition}
\begin{proof}
    \looseness=-1
    For simplicity, we consider the case where the cost set $\cC$ is a singleton, so that $\cU=\cP$.
    The extension with $r$-rectangular cost sets is straightforward.
    For any $P \in \cP^\pi$, the Bellman equation yields
    \begin{equation}\label{eq:r-rec-bellman}
        Q^\pi_{P}(s, a) = c(s, a) + \gamma \sum_{s'} P(s'\mid s, a) V^\pi_P(s') = c(s, a) + \gamma \sum^r_{i=1} \phi_i(s, a) \underbrace{\sum_{s'} w_i(s') V^\pi_P(s')}_{\fd \beta_i}\;.
    \end{equation}
    Let $V^\pi_\cU(s) \df \max_{P \in \cP} V^\pi_P(s)$, which is attained at all $s$ simultaneously by the kernel $\bar{P}$ with factors $\bar{w}_i \in \argmax_{w \in \cW_i} \sum_{s'} w(s') V^\pi_\cU(s')$ \citep{goyal2023robust}.
    For $P \in \cP^\pi$, \cref{assumption:init-dist} and $\sum_s \mu(s) V^\pi_P(s) = \sum_s \mu(s) V^\pi_\cU(s)$ give $V^\pi_P = V^\pi_\cU$, and hence $\beta_i \leq \sum_{s'} \bar{w}_i(s') V^\pi_\cU(s')$.
    Since $\phi_i(s, a) \geq 0$, \cref{eq:r-rec-bellman} gives $Q^\pi_P \leq Q^\pi_{\bar{P}}$, that is, $Q^\pi_{\bar{P}}$ is the robust action-value function.
\end{proof}

%% file: sections/discussion.tex
\section{Related Work and Open Questions}\label{sec:discussion}

\looseness=-1
This paper corrects a prevailing misconception that PSD is guaranteed to find $\varepsilon$-optimal policies in general RMDPs.
We construct an explicit counterexample in which PSD remains trapped near a suboptimal policy (\cref{sec:failure of PSD}) and identify two sufficient conditions---\cref{assm:unique-worst-case-occupancy,assm:unique-robust-value}---under which RMDPs satisfy the subgradient dominance property (\cref{sec:PSD-succeeds}).
This section situates our contributions within the existing literature on RMDPs and discusses the limitations of our results.

\paragraph{Hardness of approximate optimization.}
\looseness=-1
While the hardness of solving RMDPs has long been recognized, it has remained unclear whether finding an $\varepsilon$-optimal policy is hard even when the robust total cost and its subgradients are efficiently computable. This ambiguity may have contributed to incorrect convergence claims for PSD under general uncertainty sets \citep{wang2023policy,kitamura2025near}.

\looseness=-1
The seminal work of \citet{wiesemann2013robust} shows that evaluating the objective of general RMDPs is strongly NP-hard, but does not provide the hardness of policy optimization.
\citet{bagnell2001solving} establish NP-hardness of policy optimization only for deterministic optimal policies.
\citet{mannor2012lightning} claim hardness of approximate optimization under general uncertainty sets. However, the proof is not publicly available, and it is unclear whether the hardness is due to evaluation or optimization.
More recently, \citet{ou2025sequential} prove NP-hardness of finding an optimal policy for finite uncertainty sets, but do not address approximate optimization.

\looseness=-1
In contrast, our results (\cref{proposition:NP-hardness,proposition:NP-hard-sa-rect}) establish NP-hardness of $\varepsilon$-optimal policy identification in two settings that admit efficient computation of both the robust total cost and its subgradients:
({\rm i}) RMDPs with a finite transition set, and
({\rm ii}) RMDPs with a finite $sa$-rectangular transition set and a finite non-rectangular cost set.
Additionally, our \cref{example:RMDP-failure} demonstrates a concrete failure case of PSD in the finite uncertainty set setting, illustrating our hardness results.

\paragraph{First-order methods for RMDPs.}
Several works have established performance guarantees for first-order methods in specific classes of RMDPs. \citet{li2022first} prove global convergence of state-wise mirror descent for $sa$-rectangular uncertainty sets, while \citet{wang2022policy} analyze PSD for $R$-contamination uncertainty sets, which are also $sa$-rectangular. More recently, convex MDPs have been shown to satisfy subgradient dominance due to an underlying hidden convexity structure \citep{barakat2025global,fatkhullin2025stochastic}.

\looseness=-1
All of these settings are captured by the sufficient conditions identified in \cref{sec:PSD-succeeds}. In addition, we show that $r$-rectangular and regularized RMDPs satisfy our conditions, which were not previously known to admit subgradient dominance.

\looseness=-1
\paragraph{Open questions.}
Our sufficient conditions do not cover $s$-rectangular RMDPs, despite the existence of efficient DP-based algorithms for this setting \citep{wiesemann2013robust,grand2021scalable}.
This gap makes the $s$-rectangular case particularly intriguing: while DP can efficiently solve such RMDPs, whether PSD converges to globally optimal policies remains open.

\looseness=-1
We conjecture that $s$-rectangular RMDPs satisfy the subgradient dominance property and that dominance can hold under conditions weaker than \cref{assm:unique-worst-case-occupancy,assm:unique-robust-value}. The following proposition shows that these assumptions are not necessary.
\begin{proposition}\label{proposition:s-rectangular-subgrad-dom}
There exist $s$-rectangular RMDPs that satisfy the subgradient dominance property but satisfy neither \cref{assm:unique-worst-case-occupancy} nor \cref{assm:unique-robust-value}.
\end{proposition}
\looseness=-1
The proof is provided in \cref{sec:proof of independence}.
Identifying weaker conditions that encompass $s$-rectangular RMDPs while unifying the settings in \cref{sec:PSD-succeeds} is an interesting direction for future work.

%% file: sections/useful-lemma.tex
\section{Useful Lemmas}\label{appendix:useful-lemmas}

\begin{lemma}[The performance difference lemma \citep{kakade2002approximately}]\label{lemma:performance-difference}
  For any policies $\pi, \pi' \in \Pi$,
  $$
    J_{c, P}(\pi) - J_{c, P}(\pi') = \frac{1}{1-\gamma} \sum_{s, a} \socc{\pi'}_{P}(s) \pi'(a|s) A^{\pi}_{c, P}(s, a)\;.
  $$
\end{lemma}

\begin{lemma}[MDP smoothness {\citep[Lemma~54]{agarwal2021theory}}]\label{lemma:MDP-smooth}
  \looseness=-1
  For any $\pi, \pi' \in \Pi$,
  \begin{equation*}
    \norm*{\nabla J_{c, P}(\pi) - \nabla J_{c, P}(\pi') }_2
    \leq \frac{2 \gamma\aA}{(1-\gamma)^3}\norm*{\pi - \pi'}_2 \;.
  \end{equation*}
\end{lemma}

\begin{lemma}[Lipschitzness of RMDP objective]\label{lemma:RMDP-lipschitz}
  For any $\pi, \pi' \in \Pi$,
  $$
    \abs*{J_{\cU}(\pi) - J_{\cU}(\pi')} \leq \frac{\sqrt{\aA}}{(1-\gamma)^2} \norm*{\pi - \pi'}_2\;.
  $$
\end{lemma}
\begin{proof}
The gradient norm of $J_{c, P}(\pi)$ can be upper bounded as  
\begin{align*}
  \norm*{\nabla J_{c, P}(\pi)}_2
   & = \sqrt{\sum_{s, a} \paren*{\frac{1}{1-\gamma} \socc{\pi}_{P}(s) \qf{\pi}_{c, P}(s, a)}^2}
    \leq \frac{1}{(1-\gamma)^2}\sqrt{\sum_{s, a} \paren*{\socc{\pi}_{P}(s)}^2}
    \leq \frac{\sqrt{\aA}}{(1-\gamma)^2}\;,
\end{align*}
where the last inequality uses $\sum_n p_n^2 \leq \left(\sum_n p_n\right)^2 = 1$ for any probability vector $p \in \Delta([N])$.
The claim then follows since bounded gradients imply Lipschitzness.
\end{proof}

\begin{lemma}[A version of Danskin's theorem]\label{lemma:danskin's theorem}
  \looseness=-1
  Let $\cY \subset \R^d$ be a compact set. Let $f: \R^n \times \cY \to \R$ be a bivariate jointly continuous function and $f(\cdot, y)$ is $L$-smooth for all $y \in \cY$. Let $F(x) \df \max_{y \in \cY} f(x, y)$.
  If $\nabla_{x} f(x, \cdot)$ is continuous on $\cY$ for all $x$, then the subgradient of $F$ at $x$ is given by
  $$
    \partial F(x)=\conv\brace*{\nabla_{x} f(x, y) \given y \in \argmax_{y' \in \cY} f(x, y')} \;.
  $$
\end{lemma}
\begin{proof}
  The original Danskin's theorem \citep[Proposition~A.3.2]{bertsekas2009convex} states that if $f(\cdot, y)$ is convex for all $y \in \cY$, then the statement holds.

  \looseness=-1
  For an $L$-smooth function $f(\cdot, y)$, define $\widebar{f}(x, y) \df f(x, y) + \frac{L}{2}\norm*{x}^2_2$ and $\widebar{F}(x) \df \max_{y \in \cY} \widebar{f}(x, y)$. Since an $L$-smooth function is $L$-weakly convex \citep[Proposition~2.4]{atenas2023unified}, $\widebar{f}(\cdot, y)$ is convex. Therefore, the original Danskin's theorem indicates that:
  \begin{align*}
     & \partial \widebar{F}(x) = \conv \brace*{\nabla_{x} f(x, y) + L x \given y \in \argmax_{y' \in \cY} \widebar{f}(x, y')} \\
    \implies
     & \partial F(x) = \conv \brace*{\nabla_{x} f(x, y) \given y \in \argmax_{y' \in \cY} f(x, y')} \;,
  \end{align*}
  where the second line uses $\partial \widebar{F}(x) = \partial F(x) + L x$ \citep[Exercise~8.8]{rockafellar2009variational} and $\argmax_{y' \in \cY} \widebar{f}(x, y') = \argmax_{y' \in \cY} f(x, y')$.
\end{proof}

\begin{lemma}\label{lemma:moreau-sugrad-bound}
  Let $f: \cX \to \R$ be an $\omega$-weakly convex function where $\cX$ is a convex compact set. Define $\widebar{f}: x \in \R^d \mapsto f(x) + \delta_{\cX}(x)$.
  Let $x_\nu \df \prox_{\nu \widebar{f}}(x)$ be the shorthand for the proximal operator.
  For each $0 < \nu < 1/\omega$, there exists a subgradient $g \in \partial f(x_\nu)$ such that, for any $y \in \cX$,
  $$
    \left\langle g, x_\nu - y\right\rangle \leq \left\langle \nabla (\ME_\nu \comp \widebar{f})(x) , x_\nu - y \right \rangle\;,
  $$
  where we write $\langle x, y \rangle \df \sum_{i=1}^d x_i y_i$ for two vectors $x, y \in \R^d$.
\end{lemma}
\begin{proof}
  The Moreau envelope $\ME_\nu \comp \widebar{f}$ satisfies that, for any $x \in \R^d$,
  $$
    \paren*{\ME_\nu \comp \widebar{f}}(x)
    = \min_{x' \in \R^d} \brace*{f(x') + \delta_{\cX}(x') + \frac{1}{2\nu} \norm*{x - x'}^2_2}
    = \min_{x' \in \cX} \brace*{f(x') +\frac{1}{2\nu} \norm*{x - x'}^2_2}\;.
  $$
  It holds that $\nabla \paren*{\ME_\nu \comp \widebar{f}}(x) = \frac{1}{\nu}(x - x_\nu)$ due to \cref{eq:Moreau-grad}.
  Note that $x_\nu$ is a minimizer of a mapping $x' \mapsto f(x') + \delta_{\cX}(x') + \frac{1}{2\nu} \norm*{x-x'}^2_2$.
  Therefore,
  $$
    -\frac{1}{\nu}\paren*{x_\nu - x}
    \in \left. \partial\paren*{f(y)+\delta_{\cX}(y)}\right|_{y=x_\nu}\;.
  $$

  \looseness=-1
  It is well-known that $\partial \delta_\cX(x) = N_{\cX}(x)$, where $N_{\cX}(x)$ is the normal cone of $\cX$ at $x$ satisfying
  $$N_{\cX}(x) \df \brace*{g \in \R^d \given \langle g, y\rangle \leq \langle g, x\rangle \quad \forall y \in \cX}\;.$$
  Therefore, there exists a subgradient $g \in \partial f(x_\nu)$ such that
  $$
    -g - \frac{1}{\nu}(x_\nu - x) \in N_\cX(x_\nu)\;.
  $$
  Since any $z \in N_\cX(x_\nu)$ satisfies $\langle z, y - x_\nu \rangle \leq 0$ for any $y \in \cX$, it holds that
  $$
    \left\langle-g, y-x_\nu\right\rangle \leq\left\langle\frac{1}{\nu}\left(x_\nu - x\right), y-x_\nu\right\rangle, \quad \forall y \in \cX\;.
  $$
  Rearranging the above inequality with \cref{eq:Moreau-grad} completes the proof.
\end{proof}

\section{Technical Errors in Prior Work}\label{sec:proof-error-in-prior-work}

\looseness=-1
This section explains the technical errors in the subgradient-dominance proofs of \citet{wang2023policy} and \citet{kitamura2025near}. 

\subsection{The error in \texorpdfstring{\citet{wang2023policy}}{Wang et al.~(2023)}}

\looseness=-1
We first note that \citet{wang2023policy} consider only transition uncertainty. In addition, their Assumption 2.1 requires the cost function to be nonnegative.

\looseness=-1
A key step in their analysis is proving their Eq.\ (33), which states that, for some constant \(D>0\),
\begin{equation}\label{eq:main-goal-in-wang2023}
  J_{\cP}(\pi) - \min_{\pi'\in\Pi} J_{\cP}(\pi') \le D \norm*{\nabla (\ME_{\frac{1}{2\rho}} \comp \barJ_{\cP})(\pi)}_2\;.
\end{equation}
\looseness=-1
This corresponds to our \cref{eq:subgrad-dom-Moreau}, which ensures that any stationary point at which PSD stops is globally optimal.

\looseness=-1
Let $\tilde{\pi}=\arg \min _{\pi' \in \Pi}\left\{J_{\cP}(\pi')+\rho\|\pi-\pi'\|_2^2\right\}$ be the proximal point of $\pi$.
\citet{wang2023policy} use their Lemma D.6 to ensure that there exists $\xi \in \partial J_{\cP}(\tilde{\pi})$ such that 
$-\xi \in \cN_{\Pi}(\tilde{\pi})+2 \rho\|\tilde{\pi}-\pi\| \cdot \mathbb{B}_2$, where $\cN_{\Pi}(\tilde{\pi})$ is the normal cone of $\Pi$ at $\tilde{\pi}$ and $\mathbb{B}_2$ is the unit ball in an appropriate dimension.
Then, using the non-negativity of the cost function, Eq.\ (30)-(33) in their paper attempt to derive \cref{eq:main-goal-in-wang2023} through the following inequality: 
\begin{equation*}
0 \leq
J_{\cP}(\tilde{{\pi}})-J_{\cP}\left({\pi}^{\star}\right) 
\leq \frac{D}{1-\gamma}\underbrace{\langle\bone, -\xi \rangle}_{ \leq 0 }\leq \frac{D \sqrt{\aS\aA}}{1-\gamma}\left\|\nabla (\ME_{\frac{1}{2\rho}} \comp \barJ_{\cP})(\pi)\right\|_{2} .
\end{equation*}
However, this argument is incorrect.
Because the cost is nonnegative, every subgradient of $J_{\cP}(\cdot)$ is also non-negative by \cref{eq:J-subgradient} and \cref{eq:J-gradient}, and hence $\xi$ is a non-negative vector. As a result, the middle term is non-positive, whereas the left-hand side is non-negative.
Therefore, the inequality can hold only when $\tilde{\pi}$ is already globally optimal but does not hold in general.

\subsection{The error in \texorpdfstring{\citet{kitamura2025near}}{Kitamura et al.~(2025)}}

\looseness=-1
\citet{kitamura2025near} attempt to prove the subgradient dominance property, i.e., \cref{eq:strong-subgrad-dom}. 
Let $\cG \df \brace*{\nabla J_{c, P}(\pi) \given (c, P) \in \cU^\pi}$ be the set of gradients of the worst-case models at $\pi$, so that $\partial J_{\cU}(\pi) = \conv \cG$ by \cref{lemma:policy subgradient}. In their Eq.\ (36), using the gradient dominance lemma for MDPs (\cref{lemma:gradient dominance}), they first show that
$$
J_{\cU}(\pi) - \min_{\pi'\in\Pi} J_{\cU}(\pi') \leq D \min_{g \in \cG} \max_{\pi'\in\Pi} \sum_{s, a} (\pi(a\mid s) - \pi'(a\mid s)) g(s, a)\;.
$$
They then attempt, in their Eq.\ (37)--(39), to prove that
$$
\min_{g \in \cG} \max_{\pi'\in\Pi} \sum_{s, a} (\pi(a\mid s) - \pi'(a\mid s)) g(s, a)
= \min_{g \in \conv\cG} \max_{\pi'\in\Pi} \sum_{s, a} (\pi(a\mid s) - \pi'(a\mid s)) g(s, a)\;.
$$
using Sion's minimax theorem \citep{sion1958general} together with the convexity of \(\Pi\). However, this equality is incorrect.

\looseness=-1
More generally, their argument requires that, for two compact sets \(\cX, \cY \subseteq \mathbb{R}^d\) with \(\cX\) convex,
$$
\min _{y \in \mathcal{Y}} \max _{x \in \mathcal{X}}\langle x, y\rangle=\min _{y \in \conv(\mathcal{Y})} \max _{x \in \mathcal{X}}\langle x, y\rangle
$$
holds. However, this is false in general. For example, let $\mathcal{Y}=\{-1,1\}$ and $\mathcal{X}=[-1,1]$. Then, 
$$
1=\min _{y \in \brace{-1,1}} \max _{x \in [-1, 1]}xy
>\min _{y \in \conv(\brace{-1,1})} \max _{x \in [-1, 1]}xy=0\;,
$$
so the claimed equality does not hold.

\section{Convergence Analysis of Projected Subgradient Descent Method}\label{sec:subgrad-dom-Moreau-derivation}

\subsection{Proof of \texorpdfstring{\cref{lemma:stationary-convergence}}{Lemma \getrefnumber{lemma:stationary-convergence}}}

We utilize the existing convergence analysis of projected subgradient descent (PSD) for weakly convex functions.

\begin{lemma}[{\citep[Theorem~3.1]{davis2019stochastic}}]\label{lemma:PSD-convergence-weakly-convex}
  Let $f: \cX \to \R$ be an $\omega$-weakly convex function where $\cX$ is a convex compact set. Let $\widebar{f}: x \in \R^d \mapsto f(x) + \delta_{\cX}(x)$.
  Assume that there exist constants $L$ and $f_{\max}$ such that $\norm{g}_2 \leq L$ and $\abs{f(x)}\leq f_{\max}$ for all $x \in \cX$ and $g \in \partial f(x)$.
  Consider the update:
  $$
    x_{t+1} = \proj_{\cX}\paren*{x_t - \eta g_t}\quad \text{where } g_t \in \partial f(x_t)\;.
  $$
  Set the step size as $\eta = \frac{1}{\sqrt{T}}$.
  Then, it holds that
  $$
    \min_{t \in [T]} \norm*{\nabla (\ME_{\frac{1}{2\omega}} \comp \widebar{f})(x_t)}_2^2 \leq \frac{4 f_{\max} + 2\omega L^2}{\sqrt{T}}\;.
  $$
\end{lemma}

By substituting $f = J_{\cU}$, $\cX = \Pi$, $\omega = \rho=\frac{2\aA}{(1-\gamma)^3}$, $L=\frac{\sqrt{\aA}}{(1-\gamma)^2}$, and $f_{\max} = \frac{1}{1-\gamma}$ into \cref{lemma:PSD-convergence-weakly-convex}, we obtain \cref{lemma:stationary-convergence}.

\subsection{Derivation of \texorpdfstring{\cref{eq:subgrad-dom-Moreau}}{Eq. (\getrefnumber{eq:subgrad-dom-Moreau})}}

\looseness=-1
Here we provide a derivation of \cref{eq:subgrad-dom-Moreau} from the subgradient dominance property \cref{eq:strong-subgrad-dom}.
We use a shorthand $\pi_{\frac{1}{2\rho}} \df \prox_{\frac{1}{2\rho} \barJ_{\cU}}(\pi)$.
Let $\widebar{g} \df \nabla \paren*{\ME_{\frac{1}{2\rho}}\comp \barJ_{\cU}}(\pi)$ and let
$g_{\frac{1}{2\rho}} \in \partial J_{\cU}(\pi_{\frac{1}{2\rho}})$ be the subgradient guaranteed by \cref{lemma:moreau-sugrad-bound}.
The right-hand side of \cref{eq:strong-subgrad-dom} can be upper bounded as
\begin{align*}
  J_{\cU}(\pi_{\frac{1}{2\rho}}) - \min_{\pi'\in\Pi} J_{\cU}(\pi')
   & \le D \min_{g \in \partial J_{\cU}(\pi_{\frac{1}{2\rho}})} \max_{\pi'\in \Pi} \sum_{s, a} \paren*{\pi_{\frac{1}{2\rho}}(a\mid s) - \pi'(a\mid s)} g(s, a) \\
   & \numeq{\le}{a} D \max_{\pi'\in \Pi} \sum_{s, a} \paren*{\pi_{\frac{1}{2\rho}}(a\mid s) - \pi'(a\mid s)} g_{\frac{1}{2\rho}}(s, a) \\
   & \numeq{\le}{b} D \max_{\pi'\in \Pi} \sum_{s, a} \paren*{\pi_{\frac{1}{2\rho}}(a\mid s) - \pi'(a\mid s)} \widebar{g}(s, a) \\
   & \numeq{\leq}{c} D \max_{\pi' \in \Pi} \norm*{\pi_{\frac{1}{2\rho}} - \pi'}_2 \norm*{\widebar{g}}_2
  \numeq{\leq}{d} 2D\sqrt{\aS} \norm*{\widebar{g}}_2 \;.
\end{align*}
Here, (a) uses a particular subgradient in the minimum, (b) follows from \cref{lemma:moreau-sugrad-bound}, (c) follows from the Cauchy--Schwarz inequality, and (d) uses $\norm*{\pi' - \pi}_2 \leq 2\sqrt{\aS}$ for any $\pi', \pi \in \Pi$.

\looseness=-1
Additionally, using the Lipschitzness of $J_{\cU}(\cdot)$ (\cref{lemma:RMDP-lipschitz}), we have
\begin{align*}
  J_{\cU}(\pi) - J_{\cU}(\pi_{\frac{1}{2\rho}}) \leq \frac{\sqrt{\aA}}{(1-\gamma)^2} \norm*{\pi - \pi_{\frac{1}{2\rho}}}_2 = \frac{1-\gamma}{4\sqrt{\aA}} \norm*{\widebar{g}}_2\;,
\end{align*}
where the equality follows from \cref{eq:Moreau-grad}. By combining the above two inequalities, we obtain
\begin{align*}
  J_{\cU}(\pi) - \min_{\pi'\in\Pi} J_{\cU}(\pi')
  &\le J_{\cU}(\pi) - J_{\cU}(\pi_{\frac{1}{2\rho}})
  + J_{\cU}(\pi_{\frac{1}{2\rho}}) - \min_{\pi'\in\Pi} J_{\cU}(\pi') \\
  &\le \left(2D\sqrt{\aS} + \frac{1-\gamma}{4\sqrt{\aA}}\right) \norm*{\widebar{g}}_2\;.
\end{align*}
Setting $D' \df 2D\sqrt{\aS} + \frac{1-\gamma}{4\sqrt{\aA}}$ yields \cref{eq:subgrad-dom-Moreau}.

\section{Proof of PSD Failure (\texorpdfstring{\cref{proposition:psd-trapped-near-pi2}}{Proposition \getrefnumber{proposition:psd-trapped-near-pi2}})}\label{sec:proof-of-psd-trap}

  \begin{proposition}[Restatement of \cref{proposition:psd-trapped-near-pi2}]
    Recall the RMDP instance defined in \cref{example:RMDP-failure} with the transition uncertainty set $\cP = \{P_1, P_2\}$.
    The following two claims hold:
    \begin{itemize}
      \item The policy ${\color{red}\tilde{\pi}_2}$ that always chooses action ${\color{red}a_2}$ is a suboptimal strict local minimum of $J_{\cP}(\cdot)$ satisfying $0.505=J_{\cP}({\color{red}\tilde{\pi}_2}) > J_{\cP}({\color{blue}\tilde{\pi}_1})=0.1$.
      \item For all sufficiently small $\eta>0$, when initialized at $\pi_1 = {\color{red}\tilde{\pi}_2}$, the PSD update
      $\pi_{t+1} = \proj_{\Pi}\paren*{\pi_t - \eta g_t}$ with any $g_t \in \partial J_{\cP}(\pi_t)$ remains trapped in a neighborhood of ${\color{red}\tilde{\pi}_2}$ such that
    \[
    \|\pi_t - {\color{red}\tilde{\pi}_2}\|_\infty \leq 3\eta \qquad \forall t\ge 1\;.
    \]
    \end{itemize}
\end{proposition}
\begin{proof} 
  \looseness=-1 Throughout the proof, we use shorthands $\xi_1 \df \nabla J_{P_1}({\color{red}\tilde{\pi}_2})$ and $\xi_2 \df \nabla J_{P_2}({\color{red}\tilde{\pi}_2})$. 
  Using the symmetry that $P_2$ matches $P_1$ by swapping the states $s_1$ and $s_2$, it is straightforward to verify that
  \begin{equation}\label{eq:xi-definition}
  \begin{aligned}
    &\xi_1(s_1, {\color{blue}a_1}) = d_1 \gamma^2, \; \xi_1(s_1, {\color{red}a_2}) = d_1 \gamma, \quad \xi_1(s_2, {\color{blue}a_1}) = d_2 \gamma^2, \; \xi_1(s_2, {\color{red}a_2}) = 0\\
    & \xi_2(s_1, {\color{blue}a_1}) = d_2 \gamma^2, \; \xi_2(s_1, {\color{red}a_2}) = 0, \quad \quad \xi_2(s_2, {\color{blue}a_1}) = d_1 \gamma^2, \; \xi_2(s_2, {\color{red}a_2}) = d_1 \gamma,
  \end{aligned}
  \end{equation}
  where we defined $d_1 \df \frac{1}{1-\gamma}\socc{{\color{red}\tilde{\pi}_2}}_{P_1}(s_1) = \frac{1}{1-\gamma}\socc{{\color{red}\tilde{\pi}_2}}_{P_2}(s_2) < \frac{1}{1-\gamma}\socc{{\color{red}\tilde{\pi}_2}}_{P_1}(s_2) = \frac{1}{1-\gamma}\socc{{\color{red}\tilde{\pi}_2}}_{P_2}(s_1) \fd d_2$.

  \looseness=-1
  \paragraph{Strict local minimality.} 
  The suboptimality $0.505=J_{\cP}({\color{red}\tilde{\pi}_2}) > J_{\cP}({\color{blue}\tilde{\pi}_1})=0.1$ can be verified by direct calculations. We show the strict local minimality of ${\color{red}\tilde{\pi}_2}$.

  \looseness=-1
  Let $\pi_{x, y}({\color{blue}a_1}\mid s_1)=x$ and $\pi_{x,y}({\color{blue}a_1}\mid s_2)=y$ be a policy that perturbs ${\color{red}\tilde{\pi}_2}$ by $x$ and $y$ at states $s_1$ and $s_2$, respectively. The first-order expansion of $J_{P_1}$ and $J_{P_2}$ around ${\color{red}\tilde{\pi}_2}$ yields
  \begin{align*}
  J_{P_1}(\pi_{x,y})
  &= J_{P_1}({\color{red}\tilde{\pi}_2}) + f_1 x + f_2 y + o(x+y),\\
  J_{P_2}(\pi_{x,y})
  &= J_{P_2}({\color{red}\tilde{\pi}_2}) + f_2 x + f_1 y + o(x+y),
  \end{align*}
  where
  \begin{align*}
  f_1 &\df \xi_1(s_1, {\color{blue}a_1}) - \xi_1(s_1, {\color{red}a_2}) = -d_1\gamma(1-\gamma),\\
  f_2 &\df \xi_1(s_2, {\color{blue}a_1}) - \xi_1(s_2, {\color{red}a_2}) = d_2 \gamma^2\;.
  \end{align*}
  Note that $f_1+f_2 = d_2\gamma^2 - d_1\gamma(1-\gamma)>0$ because $d_1 < d_2$ and we set $\gamma = 0.9 > 0.5$ in \cref{example:RMDP-failure}. Since $J_{P_1}({\color{red}\tilde{\pi}_2})=J_{P_2}({\color{red}\tilde{\pi}_2})=J_{\cP}({\color{red}\tilde{\pi}_2})$, it holds that
  \begin{align*}
  J_{\cP}(\pi_{x,y})-J_{\cP}({\color{red}\tilde{\pi}_2})
  &= \max\{f_1x+f_2y,\; f_2x+f_1y\}+o(x+y)\\
  &\geq \frac{(f_1x+f_2y+f_2x+f_1y)}{2} + o(x+y) = \frac{(f_1+f_2)(x+y)}{2} + o(x+y)\;.
  \end{align*}
  Since \(f_2+f_1 > 0\) and $x, y \geq 0$, we have $J_{\cP}(\pi_{x,y})>J_{\cP}({\color{red}\tilde{\pi}_2})$ for all sufficiently small \((x,y)\neq(0,0)\). This proves that ${\color{red}\tilde{\pi}_2}$ is a strict local minimum of \(J_{\cP}\). 

  \paragraph{PSD convergence.}
  \looseness=-1
  We next prove that PSD can get trapped in a neighborhood of ${\color{red}\tilde{\pi}_2}$. 
  Note that, for any vector $(a, b)\in\R^2$ and $p \in [0, 1]$, the Euclidean projection onto $\Delta([2])$ satisfies
  \[
  \Pi_{\Delta([2])}\bigl((p,1-p)-\eta(a,b)\bigr)
  =
  \Bigl(\clip\paren*{p-\frac{\eta}{2}(a-b)},\,
  1-\clip\paren*{p-\frac{\eta}{2}(a-b)}\Bigr),
  \]
  where $\clip(u)\df \min\{1,\max\{0,u\}\}$.

  \looseness=-1
  Let $\pi_{t}({\color{blue}a_1}\mid s_1)=x_t$ and $\pi_t({\color{blue}a_1}\mid s_2)=y_t$.
  Additionally, for $g \in \R^{\aS\times \aA}$, define $F_1(g) \df g(s_1, {\color{blue}a_1}) - g(s_1, {\color{red}a_2})$ and $F_2(g) \df g(s_2, {\color{blue}a_1}) - g(s_2, {\color{red}a_2})$.
  Then, the PSD update can be expressed as
  \begin{equation}\label{eq:xy-updated}
  \begin{aligned}
  x_{t+1} = \clip\paren*{x_t - \frac{\eta}{2}F_1(g_t)} \quad\text{and}\quad y_{t+1} = \clip\paren*{y_t - \frac{\eta}{2}F_2(g_t)}.
  \end{aligned}
  \end{equation}

  \paragraph{Step 1: the initial update.}
  \looseness=-1
  Consider the first PSD update from ${\color{red}\tilde{\pi}_2}$ with $g_1 \in \partial J_{\cP}({\color{red}\tilde{\pi}_2})$.
  We show that the first PSD update lands in an $O(\eta)$-neighborhood of ${\color{red}\tilde{\pi}_2}$ that either satisfies $x_2 = 0$ or $y_2 = 0$.

  \looseness=-1
  Since $\partial J_{\cP}({\color{red}\tilde{\pi}_2}) = \conv\{\xi_1, \xi_2\}$ where $\xi_1, \xi_2$ are defined in \eqref{eq:xi-definition}, any $g_1 \in \partial J_{\cP}({\color{red}\tilde{\pi}_2})$ can be expressed as $g_1 = \lambda \xi_1 + (1-\lambda) \xi_2$ for some $\lambda \in [0, 1]$. Hence, 
  \begin{align*}
    &F_1(g_1) = \lambda F_1(\xi_1) + (1-\lambda) F_1(\xi_2) = \lambda f_1 + (1-\lambda) f_2\\
    &F_2(g_1) = \lambda F_2(\xi_1) + (1-\lambda) F_2(\xi_2) = \lambda f_2 + (1-\lambda) f_1\\
    \implies&\quad F_1(g_1) + F_2(g_1) = f_1 + f_2 > 0\;,
  \end{align*}
  where $f_1$ and $f_2$ are defined in the strict local minimality analysis.

  \looseness=-1
  Suppose that $F_1(g_1) < 0$ which implies $F_2(g_1) > 0$. Then, the update \eqref{eq:xy-updated} only changes the policy in $s_1$ and keeps the policy in $s_2$ unchanged, i.e., $y_2 = 0$. Since $F_1(g_1) \geq f_1$, the value of $x_2$ is upper bounded as
  $$
  x_2 = \clip\paren*{x_1 - \frac{\eta}{2}F_1(g_1)} \leq -\frac{\eta}{2} f_1\;.
  $$
  Similarly, if $F_1(g_1) > 0$, we have $x_2 = 0$ and $y_2 \leq -\frac{\eta}{2} f_1\;$. 
  When $F_1(g_1) \geq 0$ and $F_2(g_1) \geq 0$, the policy remains unchanged, i.e., $(x_2, y_2) = (0, 0)$. Since $F_1(g_1) + F_2(g_1) > 0$, it is impossible that $F_1(g_1) < 0$ and $F_2(g_1) < 0$ hold simultaneously. In summary, the next coordinates $(x_2, y_2)$ satisfy 
  \begin{equation*}
  (x_2, y_2) \in \brace*{(0, 0)} \cup \brace*{\brack*{0, -\frac{\eta}{2} f_1}\times \brace*{0}} \cup \brace*{\brace*{0}\times \brack*{0,-\frac{\eta}{2} f_1}}\;.
  \end{equation*}

  \paragraph{Step 2: the next update.}
  \looseness=-1
  Consider a time $t$ such that $y_t = 0$ and $x_t > 0$, i.e., the policy at state $s_2$ always chooses ${\color{red}a_2}$ while the policy at state $s_1$ has a small positive probability $x_t$ on ${\color{blue}a_1}$. 
  From the instance construction, it is clear that the worst-case transition kernel for the policy $\pi_t$ is $P_2$ and thus $g_t = \nabla J_{P_2}(\pi_t)$. 
  We show that, when $x_t = O(\eta)$, for a sufficiently small learning rate $\eta$, the next PSD update leads to a new policy with coordinates $(x_{t+1}, y_{t+1})$ satisfying $x_{t+1} = 0$ and $y_{t+1} = O(\eta)$.
  
  \looseness=-1
  Let $\varepsilon \df \frac{f_1 + f_2}{2}$.
  Since $J_{P_2}$ is smooth, there exists a constant $r_\varepsilon > 0$ such that, for all $x_t \in (0, r_\varepsilon]$, 
  \begin{equation}\label{eq:F_diff}
  -F_1(g_t) \leq -F_1(\xi_2) + \varepsilon = -f_2 + \varepsilon \quad \text{and}\quad -F_2(g_t) \leq -F_2(\xi_2) + \varepsilon = -f_1 + \varepsilon\;.
  \end{equation}
  We now set $\eta \leq \frac{2r_\varepsilon}{(-f_1 + \varepsilon)}$ and assume the value of $x_t$ to satisfy \cref{eq:eta-xt-condition}:
  \begin{equation}\label{eq:eta-xt-condition}
   x_t \leq \frac{\eta}{2}(-f_1 + \varepsilon) \leq r_\varepsilon\;.
  \end{equation}
  Then, the next PSD update leads to a new policy with coordinates $(x_{t+1}, y_{t+1})$ satisfying
  \begin{align*}
  x_{t+1} &= \clip\paren*{x_t - \frac{\eta}{2}F_1(g_t)} \leq
  \clip\paren*{\frac{\eta}{2}(-f_1 + \varepsilon) - \frac{\eta}{2}(f_2 - \varepsilon)}=0 \\
  \text{and}\quad y_{t+1} &= \clip\paren*{0 - \frac{\eta}{2}F_2(g_t)} \leq \frac{\eta}{2} (-f_1 + \varepsilon),
  \end{align*}
  where the inequalities use \cref{eq:F_diff}. 
  
  \looseness=-1
  By symmetry, if $x_t = 0$ and $y_t = \delta_y \leq \frac{\eta}{2}(-f_1 + \varepsilon)$, we can similarly show that the next PSD update leads to a new policy with coordinates $(x_{t+1}, y_{t+1})$ satisfying 
  $$x_{t+1} \leq \frac{\eta}{2}(-f_1 + \varepsilon) \quad \text{and}\quad y_{t+1} = 0\;.
  $$

  \paragraph{Step 3: repeating argument.}
  Consider a sufficiently small learning rate $\eta \leq \frac{2r_\varepsilon}{(-f_1 + \varepsilon)}$ where $\varepsilon =\frac{f_1 + f_2}{2}$ and $r_\varepsilon$ is defined to satisfy \cref{eq:F_diff}.
  Then, using the results from \textbf{Step 1}, the first PSD update from ${\color{red}\tilde{\pi}_2}$ leads to a new policy with coordinates $(x_2, y_2)$ satisfying 
  $$
    (x_2, y_2) \in \brace*{(0, 0)} \cup \brace*{\brack*{0, -\frac{\eta}{2} f_1}\times \brace*{0}} \cup \brace*{\brace*{0}\times \brack*{0,-\frac{\eta}{2} f_1}}\;.
  $$

  \looseness=-1
  Since this $(x_2, y_2)$ satisfies the conditions for $(x_t, y_t)$ in \textbf{Step 2} (see \cref{eq:eta-xt-condition} for $x_t$) or $(x_t, y_t)=(0, 0)$ as in \textbf{Step 1}, the next PSD update leads to coordinates $(x_3, y_3)$ satisfying 
  $$
    (x_3, y_3) \in \brace*{(0, 0)} \cup \brace*{\brack*{0, \frac{\eta}{2} (-f_1 + \varepsilon)}\times \brace*{0}} \cup \brace*{\brace*{0}\times \brack*{0,\frac{\eta}{2} (-f_1 + \varepsilon)}}\;.
  $$

  \looseness=-1
  This $(x_3, y_3)$ again satisfies the conditions for $(x_t, y_t)$ in \textbf{Step 2} or \textbf{Step 1}.
  The claim follows by iterating the above argument for all $t \geq 2$ and noting that
  $$
  \frac{\eta}{2} (-f_1 + \varepsilon) = \frac{\eta}{2} \frac{-f_1 + f_2}{2} 
  = \frac{\eta}{4} \paren*{d_2\gamma^2 + d_1\gamma(1-\gamma)} 
  \leq \frac{\eta}{4 \cdot 0.1}0.9 \leq 3 \eta\;,
  $$
  where we used the fact that $d_1 < d_2 < \frac{1}{1-\gamma}$ and $\gamma = 0.9$ in \cref{example:RMDP-failure}.
 \end{proof}

\section{Proofs of NP-hardness}\label{sec:proof-of-NP-hard}

In this section, we provide proofs of \cref{proposition:NP-hardness} and \cref{proposition:NP-hard-sa-rect}.

\subsection{NP-hardness of finite \texorpdfstring{$\cP$ and singleton $\cC$}{P and singleton C}}

\begin{proposition}[Restatement of \cref{proposition:NP-hardness}]\label{proposition:NP-hardness-restated}
  For any fixed $\gamma\in(0,1)$ and any $\varepsilon < 0.5 \gamma^3 (1-\gamma)^{-1}$, finding an $\varepsilon$-optimal policy in RMDPs with a finite transition uncertainty set is NP-hard.
\end{proposition}

\begin{proof}
We prove the claim by a reduction from 3-SAT.

\begin{definition}[3-SAT decision problem]
A 3-SAT instance consists of $N$ Boolean variables $(x_1, \dots, x_{N}) \in \{ \text{False}, \text{True} \}^N$ and $M$ clauses $C_1, \dots, C_M$, where each clause is a disjunction of three literals. A literal is either a variable $x_n$ or its negation $\neg x_{n}$. For example, a clause $C_m$ may take the form 
$$
C_{m} = (x_i \vee \neg x_j \vee x_k)\;.
$$
The 3-SAT decision problem asks whether there exists an assignment $(x_1, \dots, x_{N}) \in \{ \text{False}, \text{True} \}^N$ such that the entire formula evaluates to True:
$$ C_{1} \wedge C_{2} \wedge \cdots \wedge C_{M} = \text{True}
$$
\end{definition}

\paragraph{RMDP construction.}
\looseness=-1
Without loss of generality, no clause contains both $x_n$ and $\neg x_n$, since such a clause is always satisfied.

\looseness=-1
We construct an RMDP instance corresponding to a given 3-SAT instance. We begin by introducing the following states:

\begin{itemize}
\item the initial state $s_{\text{ini}}$;
\item the clause states $s_{C_1}, \ldots, s_{C_M}$;
\item the variable states $s_{x_1}, \ldots, s_{x_N}$; and
\item the absorbing states $s_0$ and $s_{+1}$, where the agent incurs $+1$ cost in $s_{+1}$.
\end{itemize}

\looseness=-1
We set the cost to $0$ for all state--action pairs except for the absorbing state $s_{+1}$, and the initial distribution to $\mu(s_{\text{ini}}) = 1$.

\looseness=-1
For each clause $C_m$, we define a deterministic transition kernel $P_m$. Under $P_m$, the agent transitions from the initial state to the corresponding clause state $s_{C_m}$:
$$
P_{m}(s_{C_{m}} \mid s_{\text{ini}})=1
$$

\looseness=-1
Let $m_1, m_2, m_3$ denote the indices of the three literals appearing in clause $C_m$. For example, $C_m = (x_{m_1} \vee \neg x_{m_2} \vee x_{m_3})$. At the clause state $s_{C_m}$, the agent chooses an action from the set $\{a_1, a_2, a_3\}$. Selecting action $a_i \in \{a_1, a_2, a_3\}$ deterministically transitions the agent to the corresponding variable state $s_{x_{m_i}}$.

\looseness=-1
At a variable state $s_{x_n}$, the agent has two available actions, $a_{\text{F}}$ and $a_{\text{T}}$.
Under $P_m$, if $x_n$ appears in $C_m$ as the positive literal $x_n$, then action $a_{\text{T}}$ leads to the absorbing state $s_0$, while action $a_{\text{F}}$ transitions to the cost state $s_{+1}$.
If $x_n$ appears as the negated literal $\neg x_n$, these transitions are reversed: action $a_{\text{F}}$ leads to $s_0$, while action $a_{\text{T}}$ transitions to $s_{+1}$.
\cref{fig:three-sat-rmdp} illustrates the resulting transition structure under $P_m$ for the clause $C_m = (x_{m_1} \vee \neg x_{m_2} \vee x_{m_3})$.

\begin{figure}[t]
  \begin{center}
  \includegraphics[width=\textwidth]{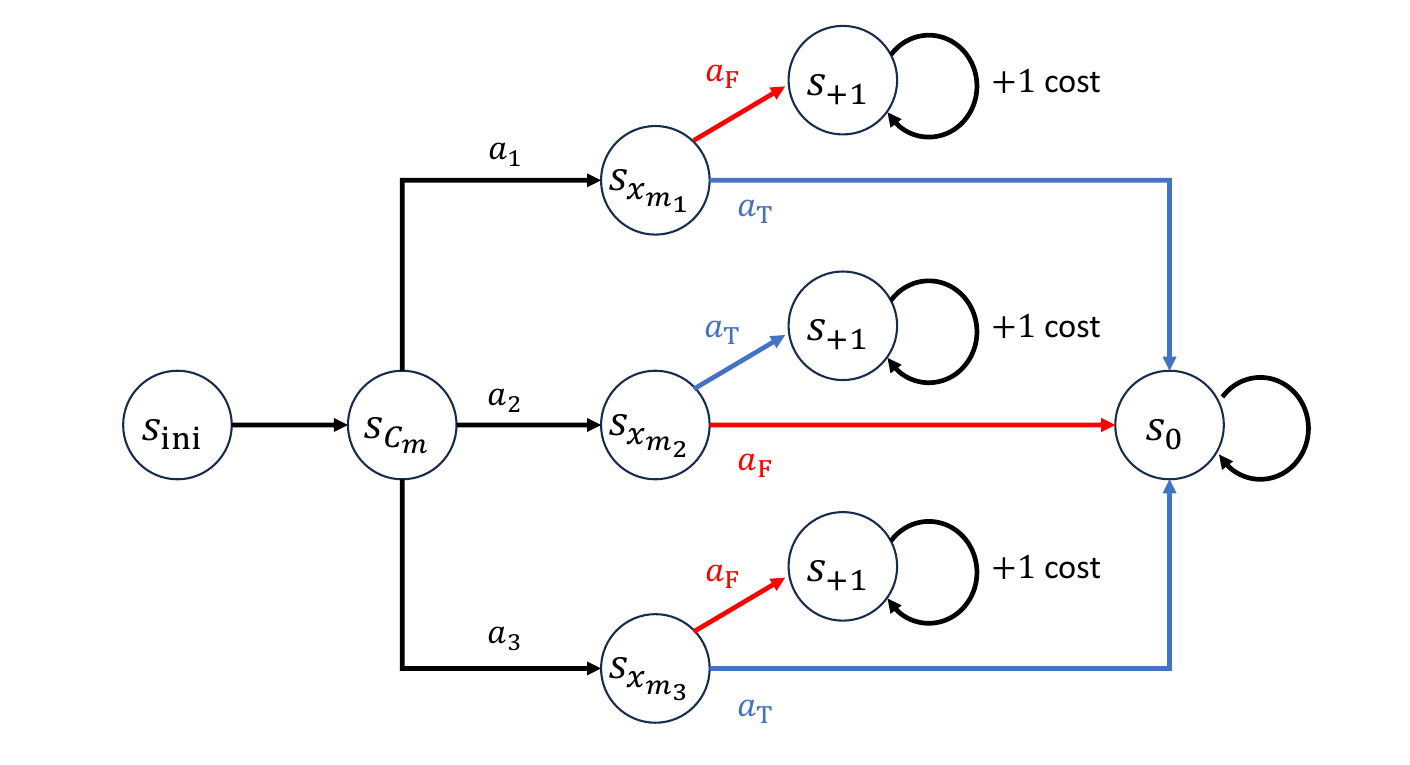}
  \caption{An illustration of the transition kernel $P_m$ for a clause $C_m = (x_{m_1} \vee \neg x_{m_2} \vee x_{m_3})$.}
  \label{fig:three-sat-rmdp}
  \end{center}
\end{figure}

\paragraph{Policy construction.}
We prove the claim by showing that the optimal robust total cost of this RMDP is $0$ when the 3-SAT instance is satisfiable, and is at least $0.5\gamma^{3}(1-\gamma)^{-1}$ otherwise.

\looseness=-1
Since the agent reaches a clause state $s_{C_m}$ only under the transition kernel $P_m$, greedily selecting an action at $s_{C_m}$ always minimizes the robust cost among all policies at that state. Therefore, without loss of generality, we restrict our attention to the policies that satisfy
$$
\pi(a_{i} \mid s_{C_{m}}) =1 \quad\text{where}\quad i \in \arg\min_{j \in \{ 1, 2, 3 \}} V^{\pi}_{P_{m}}(s_{x_{m_{j}}}) \text{, breaking ties arbitrarily.}
$$

\paragraph{When 3-SAT is unsatisfiable.} Using a policy $\pi$, we construct a 3-SAT assignment $\beta_{\pi}: \{ 1,\cdots N \}\to \{ \text{False}, \text{True} \}$ to the variables $x_1 \cdots x_N$ by
$$
\beta_{\pi}(n) = 
\begin{cases}
\text{True} \quad &\text{if } \pi(a_{\text{T}} | s_{x_{n}}) > 0.5\\
\text{False} \quad &\text{otherwise }
\end{cases}
$$
Under $P_m$, the value $V^\pi_{P_m}(s_{x_{m_i}})$ equals $\gamma (1-\gamma)^{-1} \pi(a_{\text{F}} \mid s_{x_{m_i}})$ if the $i$-th literal is $x_{m_i}$, and $\gamma (1-\gamma)^{-1} \pi(a_{\text{T}} \mid s_{x_{m_i}})$ if it is $\neg x_{m_i}$.
Hence, for a clause $C_m$, we have the following two cases:
\begin{itemize}
  \item If $i$-th literal is $x_{m_i}$, $\beta_{\pi}$ satisfies the $i$-th literal if and only if $V^\pi_{P_{m}}(s_{x_{m_{i}}}) < 0.5 \gamma (1-\gamma)^{-1}$.
  \item If $i$-th literal is $\neg x_{m_i}$, $\beta_\pi$ satisfies the $i$-th literal if and only if $V^\pi_{P_{m}}(s_{x_{m_{i}}}) \leq 0.5 \gamma (1-\gamma)^{-1}$.
\end{itemize}
Since the 3-SAT instance is unsatisfiable, for any policy $\pi$ there exists a clause $C_m$ for which all three literals are unsatisfied by $\beta_{\pi}$, and hence $V^\pi_{P_m}(s_{x_{m_i}}) \geq 0.5 \gamma (1-\gamma)^{-1}$ for every $i \in \{1, 2, 3\}$.
Under $P_m$, the agent moves from $s_{\text{ini}}$ to $s_{C_m}$ and then to one of the three variable states without incurring any cost. Consequently,
$$
J_{\cP}(\pi) \geq J_{P_m}(\pi) = \gamma^2 \sum_{i=1}^3 \pi(a_i \mid s_{C_m}) V^\pi_{P_m}(s_{x_{m_i}}) \geq 0.5 \gamma^3 (1-\gamma)^{-1}
\quad\text{for all}\quad \pi \in \Pi\;.
$$

\paragraph{When 3-SAT is satisfiable.} Let $\beta$ be an assignment that satisfies the 3-SAT. Define a policy by 
$$
\pi_{\beta}(a_{\text{T}} \mid s_{x_{n}}) = 
\begin{cases}
& 1 \quad \text{if } \beta(n) = \text{True}\\
&0 \quad \text{otherwise}
\end{cases}
$$
Since $\beta$ satisfies the 3-SAT, for any clause $C_m$, this policy attains $V^{\pi_{\beta}}_{P_{m}}(s_{x_{m_{i}}})=0$ for at least one $i\in \{ 1, 2, 3 \}$. Due to the greedy action selection at the clause state $s_{C_{m}}$, the total cost under $P_m$ becomes $J_{P_{m}}(\pi_{\beta})=0$. Consequently,
$$
\min_{\pi \in \Pi} J_{\cP}(\pi) = 0\;.
$$

\looseness=-1
Therefore, if one can find an $\varepsilon$ optimal policy $\pi_{\varepsilon}$ with $\varepsilon< 0.5\gamma^3 (1-\gamma)^{-1}$, we can solve the 3-SAT problem: It is satisfiable if $J_{\cP}(\pi_{\varepsilon}) < 0.5 \gamma^3 (1-\gamma)^{-1}$, unsatisfiable otherwise.
\end{proof}

\subsection{NP-hardness of \texorpdfstring{$sa$-rectangular $\cP$ and finite non-rectangular $\cC$}{sa-rectangular P and finite non-rectangular C}}

\begin{proposition}[First claim of \cref{proposition:NP-hard-sa-rect}]\label{proposition:NP-hard-sa-rect-restated}
  For any fixed $\gamma\in(0,1)$ and any $\varepsilon < \frac{\gamma^2}{2(2-\gamma)}$, finding an $\varepsilon$-optimal policy with an $sa$-rectangular finite transition set and a finite non-rectangular cost set is NP-hard.
\end{proposition}

  \begin{figure}[t]
    \begin{center}
    \includegraphics[width=\textwidth]{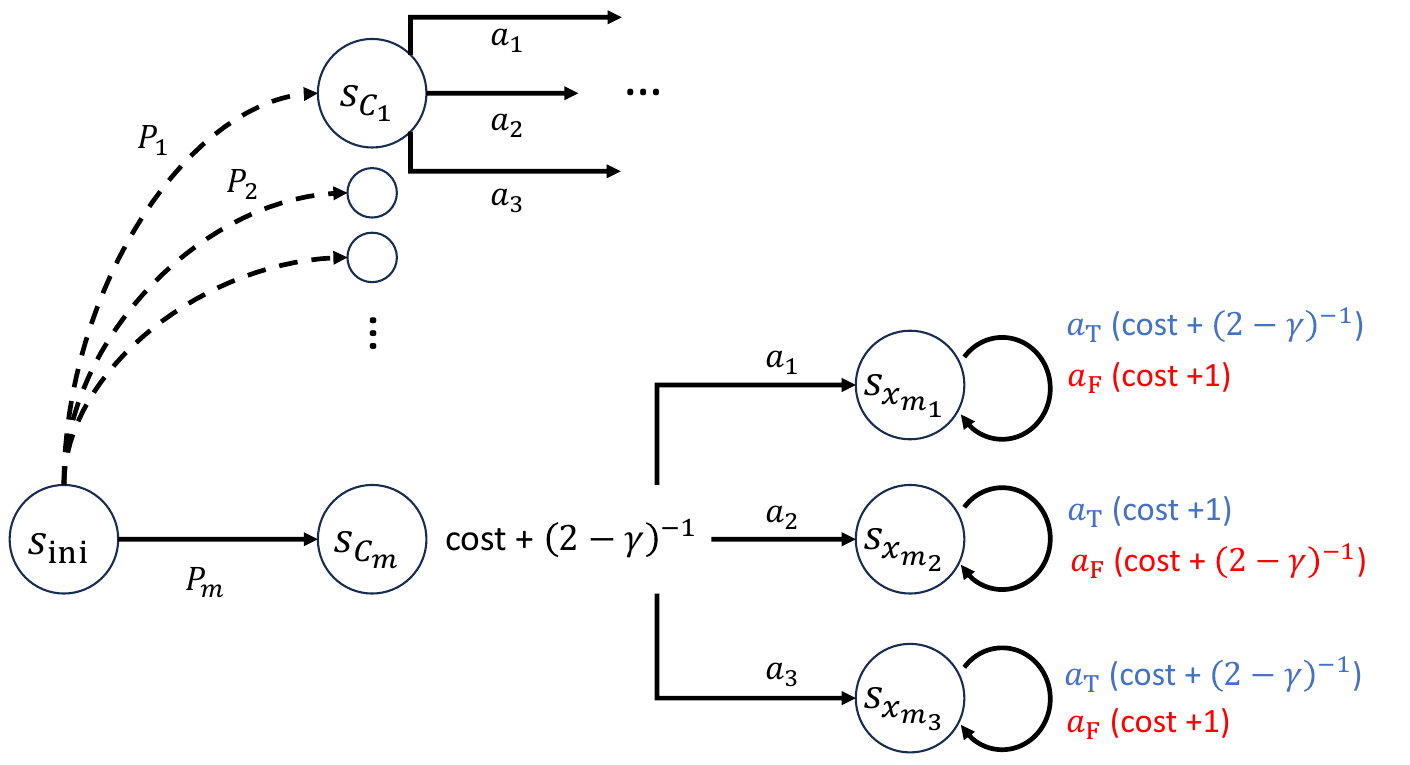}
    \caption{An illustration of the transition and the $m$-th cost function for $C_m = (x_{m_1} \vee \neg x_{m_2} \vee x_{m_3})$.
    For any policy, the total cost starting from $s_{C_m}$ is at least $((2-\gamma)(1-\gamma))^{-1}$, whereas that starting from $s_{C_i}$ with $i\neq m$ is at most $\gamma(1-\gamma)^{-1}$. Since $(2-\gamma)^{-1}>\gamma$, the worst-case transition kernel for $c_m$ is $P_m$.
    }
    \label{fig:sa-three-sat-rmdp}
    \end{center}
  \end{figure}

\begin{proof}
  We prove the claim by constructing an RMDP instance which is almost equivalent to that in \cref{proposition:NP-hardness-restated}.
  Consider the following state space:
  \begin{itemize}
  \item the initial state $s_{\text{ini}}$;
  \item the clause states $s_{C_1}, \ldots, s_{C_M}$; and 
  \item the variable states $s_{x_1}, \ldots, s_{x_N}$.
  \end{itemize}
  The action space and the initial distribution $\mu(s_{\text{ini}}) = 1$ are identical to those in \cref{proposition:NP-hardness-restated}.

  \looseness=-1
  We define the $sa$-rectangular transition uncertainty set $\cP = \times_{s,a} \cP_{s,a}$ as follows:
  \begin{itemize}
    \item Each variable state $s_{x_n}$ is absorbing, that is, any action taken at $s_{x_n}$ transitions to $s_{x_n}$ with probability $1$.
    \item Transitions at the clause states $s_{C_m}$ are the same as in \cref{proposition:NP-hardness-restated}: choosing an action $a_i \in \{a_1, a_2, a_3\}$ deterministically transitions the agent to the corresponding variable state $s_{x_{m_i}}$.
    \item At the initial state $s_{\text{ini}}$, the adversary chooses the next clause state independently for each action, that is, $\cP_{s_{\text{ini}}, a} \df \{p_1, \ldots, p_M\}$ for every action $a$, where $p_m \in \Delta(\S)$ puts probability one on $s_{C_m}$.
  \end{itemize}
  The sets $\cP_{s,a}$ at all other state--action pairs are singletons, so $\cP$ is a finite $sa$-rectangular set. For each $m \in [M]$, let $P_m \in \cP$ denote the kernel that selects $p_m$ for every action at $s_{\text{ini}}$.

  \looseness=-1
  We next define the finite cost set $\cC = \brace{c_1, \dots, c_M}$ as follows. For each $m \in [M]$, 
  \begin{itemize}
    \item At clause state $s_{C_i}$, we set $c_m(s_{C_i}, \cdot) = 0$ if $i \neq m$, and $c_m(s_{C_m}, \cdot) = (2-\gamma)^{-1}$.
    \item At variable states $s_{x_{m_i}}$ for $i=1, 2, 3$, we set $c_m(s_{x_{m_i}}, a_{\text{T}}) = (2-\gamma)^{-1}$ and $c_m(s_{x_{m_i}}, a_{\text{F}}) = 1$ if the $i$-th literal is $x_{m_i}$; otherwise, we set $c_m(s_{x_{m_i}}, a_{\text{T}}) = 1$ and $c_m(s_{x_{m_i}}, a_{\text{F}}) = (2-\gamma)^{-1}$.
    \item All remaining costs are set to zero.
  \end{itemize}
  \cref{fig:sa-three-sat-rmdp} illustrates the resulting transition structure and cost function for the clause $C_m = (x_{m_1} \vee \neg x_{m_2} \vee x_{m_3})$.

  \looseness=-1
  Note that, for the cost function $c_m$, the worst-case transition kernel is $P_m$. Indeed, for every policy $\pi$,
  $$
  V^\pi_{c_m,P_m}(s_{C_m})
  \geq \frac{1}{(2-\gamma)(1-\gamma)}\;,
  \qquad
  V^\pi_{c_m,P_i}(s_{C_i})
  \leq \frac{\gamma}{1-\gamma}
  \quad \forall i\neq m\;.
  $$
  Since $(2-\gamma)^{-1}>\gamma$ and $s_{\text{ini}}$ is never revisited, transitioning to $s_{C_m}$ under every action at $s_{\text{ini}}$ is worst-case for $c_m$. Consequently, for $\cU\df\cC\times\cP$, we have
  $$
  \max_{m \in [M]} J_{c_m, P_m}(\pi) = J_\cU(\pi) \quad \forall \pi \in \Pi\;.
  $$
  Therefore, the constructed RMDP has the same robust objective as the paired finite uncertainty set $\{(c_m, P_m) \mid m \in [M]\}$.

  \looseness=-1
  Under the $m$-th model $(c_m,P_m)$, the total cost of any policy is obtained from that under the corresponding model in \cref{proposition:NP-hardness-restated} by adding $\gamma/((2-\gamma)(1-\gamma))$ and scaling its policy-dependent part by $(1-\gamma)/(\gamma(2-\gamma))$.
  Indeed, the baseline cost $(2-\gamma)^{-1}$ at $s_{C_m}$ and at the absorbing variable state yields the policy-independent shift
  $$
  \frac{\gamma}{2-\gamma}
  +\frac{\gamma^2}{(2-\gamma)(1-\gamma)}
  =\frac{\gamma}{(2-\gamma)(1-\gamma)}\;,
  $$
  while the policy-dependent cost is reduced to $1-(2-\gamma)^{-1}=(1-\gamma)/(2-\gamma)$ and incurred one step earlier, yielding the stated scaling factor.
  Hence, by the same argument as in \cref{proposition:NP-hardness-restated}, finding an $\varepsilon$-optimal policy with $\varepsilon<\gamma^2/(2(2-\gamma))$ is NP-hard: the instance is satisfiable if
  $$
  J_{\cU}(\pi_\varepsilon)
  <\frac{\gamma}{(2-\gamma)(1-\gamma)}
  +\frac{\gamma^2}{2(2-\gamma)}
  $$
  and is unsatisfiable otherwise.
\end{proof}

\begin{proposition}[Second claim of \cref{proposition:NP-hard-sa-rect}]\label{proposition:NP-hard-sa-rect-restated2}
  Deciding whether an RCMDP with $sa$-rectangular $\cP$ is feasible is NP-hard.
\end{proposition}
\begin{proof}
  Consider the $sa$-rectangular transition set $\cP$ and cost functions $\cC = \brace{c_1, \dots, c_M}$ constructed in the proof of \cref{proposition:NP-hard-sa-rect-restated}.
  We consider the following RCMDP problem:
\begin{align*}
   \text{(RCMDP)} \quad \min_{\pi \in \Pi} J_{c_0, \cP}(\pi) \quad \text{such that} \quad \max_{m \in [M]} J_{c_m, \cP}(\pi) \leq b
   \quad\text{where}\quad b \df \frac{\gamma}{(2-\gamma)(1-\gamma)}\;,
\end{align*}
where $c_0 \equiv 0$ is the objective cost and $c_1, \ldots, c_M$ are the constraint costs constructed above.
The proof of \cref{proposition:NP-hard-sa-rect-restated} shows that $\max_{m \in [M]} J_{c_m, \cP}(\pi) = J_{\cU}(\pi)$ for every $\pi \in \Pi$, and that $\min_{\pi \in \Pi} J_{\cU}(\pi)$ equals $b$ if the 3-SAT instance is satisfiable and is at least $b + \gamma^2 / (2(2-\gamma))$ otherwise.
Hence, the RCMDP is feasible if and only if the 3-SAT instance is satisfiable.
\end{proof}

\section{Proof of \texorpdfstring{\cref{proposition:independent-conditions,proposition:s-rectangular-subgrad-dom}}{Propositions \getrefnumber{proposition:independent-conditions} and \getrefnumber{proposition:s-rectangular-subgrad-dom}}}\label{sec:proof of independence}

\begin{proposition}[Restatement of \cref{proposition:independent-conditions}]
  Neither of the following conditions implies the other:
  \begin{itemize}
    \item (\cref{assm:unique-worst-case-occupancy}) Let $\cP^\pi \df \brace{P \given (c, P) \in \cU^\pi}$ be the set of worst-case transition kernels under policy $\pi$. The set $\cP^\pi$ is a singleton for any $\pi \in \Pi$.
    \item (\cref{assm:unique-robust-value}) For any $\pi \in \Pi$, the set $\cQ^\pi \df \brace{Q^\pi_{c, P} \given (c, P) \in \cU^\pi}$ of worst-case action-value functions contains a robust action-value function $Q^\pi_\cU$, that is, $Q \leq Q^\pi_\cU$ for all $Q \in \cQ^\pi$.
  \end{itemize}
\end{proposition}
\begin{proof}
  \looseness=-1
  \textbf{Example 1: unique worst-case action-value, non-unique worst-case transition.}
  Consider an RMDP with three states $\{s_1,s_2,s_3\}$ and a single action $\{a_1\}$.
  The initial state is $s_1$, and the cost function is identically zero.
  The uncertainty set consists of two transition kernels:
  $$
  P_1: P_1(s_2 \mid s_1, a_1) = 1\;,\qquad
  P_2: P_2(s_3 \mid s_1, a_1) = 1\;.
  $$
  Since the cost is zero everywhere, the action-value function is identically zero under both $P_1$ and $P_2$ for every policy $\pi$, so $\cQ^\pi$ is a singleton.
  However, both $P_1$ and $P_2$ are worst-case transitions, so $\cP^\pi = \{P_1,P_2\}$.
  This instance satisfies only the second condition.

  \looseness=-1
  \paragraph{Example 2: unique worst-case transition, no robust action-value function.}
  Now consider a cost-robust MDP with two states $\{s_1,s_2\}$ and two actions $\{a_1,a_2\}$.
  The initial state is $s_1$.
  After taking any action at $s_1$, the agent deterministically transitions to $s_2$ and remains there forever.
  The transition kernel is therefore unique.

  \looseness=-1
  Define two cost functions $c_1$ and $c_2$ as
  $$
  c_1: c_1(s_1, a_1) = 0\;,\; c_1(s_1, a_2) = 1\;,\qquad
  c_2: c_2(s_1, a_1) = 1\;,\; c_2(s_1, a_2) = 0\;,
  $$
  and assume zero cost at state $s_2$ for both cost functions.
  Clearly, the worst-case transition kernel is unique for every policy, so $\abs{\cP^\pi}=1$.

  \looseness=-1
  However, a robust action-value function does not always exist.
  Consider the policy $\pi$ satisfying $\pi(a_1 \mid s_1) = \pi(a_2 \mid s_1) = 1/2$.
  Under this policy, both $c_1$ and $c_2$ attain the same worst-case value, and hence both belong to $\cU^\pi$.
  The corresponding action-value functions satisfy
  \begin{align*}
  &Q^{\pi}_{c_1}(s_1, a_1) = 0\;, \; Q^{\pi}_{c_1}(s_1, a_2) = 1\;,\\
  &Q^{\pi}_{c_2}(s_1, a_1) = 1\;, \; Q^{\pi}_{c_2}(s_1, a_2) = 0\;,
  \end{align*}
  so neither dominates the other and no robust action-value function exists under $\pi$.
  This instance satisfies only the first condition.
\end{proof}

\begin{figure}[t]
  \begin{center}
  \includegraphics[width=0.8\linewidth]{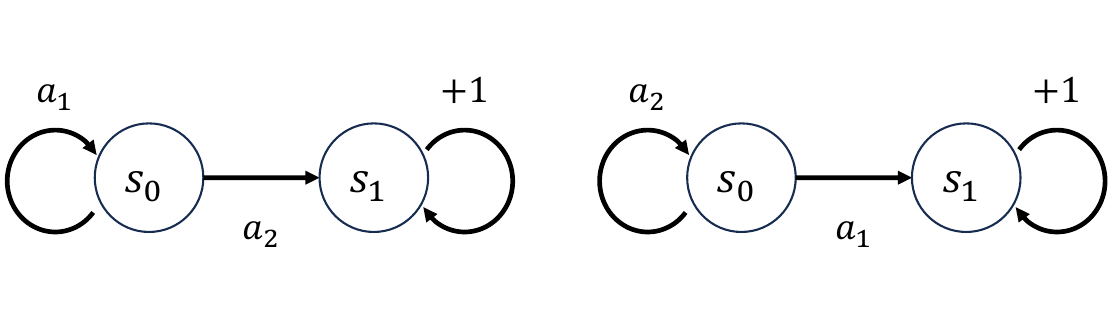}
  \caption{An $s$-rectangular RMDP example satisfying the subgradient dominance property but satisfying neither \cref{assm:unique-worst-case-occupancy} nor \cref{assm:unique-robust-value}. Left: $P_1$. Right: $P_2$.}
  \label{fig:s-rect-subgrad-dom}
  \end{center}
\end{figure}

\begin{proposition}[Restatement of \cref{proposition:s-rectangular-subgrad-dom}]\label{proposition:s-rectangular-subgrad-dom-restated}
There exist $s$-rectangular RMDPs that satisfy the subgradient dominance property but satisfy neither \cref{assm:unique-worst-case-occupancy} nor \cref{assm:unique-robust-value}.
\end{proposition}
\begin{proof}
 Consider the RMDP illustrated in \cref{fig:s-rect-subgrad-dom} with two states $\{s_0,s_1\}$ and two actions $\{a_1,a_2\}$. 
 State $s_0$ is the initial state and $s_1$ is absorbing.
 State $s_1$ incurs $+1$ cost and $s_0$ incurs $0$ cost.
 
 \looseness=-1
 It has two transition kernels $P_1$ and $P_2$ defined as 
 \begin{align*}
 P_1:\quad &P_1(s_0 \mid s_0, a_1) = 1\;,\quad P_1(s_1 \mid s_0, a_2) = 1\;,\\
 P_2:\quad &P_2(s_1 \mid s_0, a_1) = 1\;,\quad P_2(s_0 \mid s_0, a_2) = 1\;.
 \end{align*}
 The uncertainty set $\cU = \brace{P_1, P_2}$ is $s$-rectangular since the uncertainty arises only at state $s_0$.

 \looseness=-1
 Fix any $\gamma \in (0,1)$, and let $\pi_p$ denote a policy satisfying $\pi_p(a_1 \mid s_0)=p$ and $\pi_p(a_2 \mid s_0)=1-p$ for $p \in [0,1]$. The action choice at the absorbing state $s_1$ is immaterial. Since $V^{\pi_p}_{P_i}(s_1)=1/(1-\gamma)$, the Bellman equations at $s_0$ give
 \begin{align*}
 V^{\pi_p}_{P_1}(s_0)
 &= \gamma\left[pV^{\pi_p}_{P_1}(s_0)+\frac{1-p}{1-\gamma}\right],\\
 V^{\pi_p}_{P_2}(s_0)
 &= \gamma\left[\frac{p}{1-\gamma}+(1-p)V^{\pi_p}_{P_2}(s_0)\right].
 \end{align*}
 Consequently,
 $$
 J_{P_1}(\pi_p)=\frac{\gamma(1-p)}{(1-\gamma)(1-\gamma p)}\;, \qquad
 J_{P_2}(\pi_p)=\frac{\gamma p}{(1-\gamma)(1-\gamma+\gamma p)}\;.
 $$
 Their difference is
 $$
 J_{P_1}(\pi_p)-J_{P_2}(\pi_p)
 =\frac{\gamma(1-2p)}{(1-\gamma p)(1-\gamma+\gamma p)}\;.
 $$
 The first branch is strictly decreasing and the second is strictly increasing. Thus, writing $\Phi(p) \df J_{\cU}(\pi_p)$, the active model is $P_1$ for $p<1/2$ and $P_2$ for $p>1/2$. The two branches intersect only at the unique minimizer $p^\star=1/2$, where
 $$
 \Phi(p^\star)=\frac{\gamma}{(1-\gamma)(2-\gamma)}\;.
 $$

 \looseness=-1
 We next verify the subgradient dominance property directly. For a policy subgradient $g$, only the difference $\delta \df g(s_0,a_1)-g(s_0,a_2)$ contributes to the Frank--Wolfe gap in \cref{eq:strong-subgrad-dom}, since the two components at $s_1$ are equal. The set of possible differences is exactly $\partial\Phi(p)$, and hence
 $$
 G(\pi_p)
 =\min_{\delta\in\partial\Phi(p)}\max_{q\in[0,1]}(p-q)\delta
 =\begin{cases}
 \dfrac{\gamma(1-p)}{(1-\gamma p)^2},&p<1/2,\\[1mm]
 0,&p=1/2,\\[1mm]
 \dfrac{\gamma p}{(1-\gamma+\gamma p)^2},&p>1/2.
 \end{cases}
 $$
 Here, at $p=1/2$, the interval
 $$
 \partial\Phi(1/2)=\left[-\frac{4\gamma}{(2-\gamma)^2},\frac{4\gamma}{(2-\gamma)^2}\right]
 $$
 contains zero. On the other hand, the optimality gap is
 $$
 \Phi(p)-\Phi(1/2)
 =\begin{cases}
 \dfrac{\gamma(1-2p)}{(1-\gamma p)(2-\gamma)},&p<1/2,\\[1mm]
 0,&p=1/2,\\[1mm]
 \dfrac{\gamma(2p-1)}{(1-\gamma+\gamma p)(2-\gamma)},&p>1/2.
 \end{cases}
 $$
 For $p<1/2$ and $p>1/2$, respectively, the ratios of this gap to $G(\pi_p)$ satisfy
 $$
 \frac{(1-2p)(1-\gamma p)}{(2-\gamma)(1-p)} \leq 1\;, \qquad
 \frac{(2p-1)(1-\gamma+\gamma p)}{(2-\gamma)p} \leq 1\;.
 $$
 Therefore, $J_{\cU}(\pi_p)-\min_{\pi'\in\Pi}J_{\cU}(\pi')\leq G(\pi_p)$, so the robust objective is subgradient-dominant with constant $D=1$.

 \looseness=-1
 Finally, at $p=1/2$, both $P_1$ and $P_2$ are worst-case transition kernels, so \cref{assm:unique-worst-case-occupancy} is not satisfied. Writing $v_0 \df \gamma/((1-\gamma)(2-\gamma))$ and $v_1 \df 1/(1-\gamma)$, their action-value functions satisfy
 \begin{align*}
 Q^{\pi_{1/2}}_{P_1}(s_0,a_1)&=\gamma v_0=\frac{\gamma^2}{(1-\gamma)(2-\gamma)}, &
 Q^{\pi_{1/2}}_{P_1}(s_0,a_2)&=\gamma v_1=\frac{\gamma}{1-\gamma},\\
 Q^{\pi_{1/2}}_{P_2}(s_0,a_1)&=\gamma v_1=\frac{\gamma}{1-\gamma}, &
 Q^{\pi_{1/2}}_{P_2}(s_0,a_2)&=\gamma v_0=\frac{\gamma^2}{(1-\gamma)(2-\gamma)}.
 \end{align*}
 Since $v_0 < v_1$ for $\gamma \in (0,1)$, we have $Q^{\pi_{1/2}}_{P_1}(s_0,a_1) < Q^{\pi_{1/2}}_{P_2}(s_0,a_1)$ and $Q^{\pi_{1/2}}_{P_1}(s_0,a_2) > Q^{\pi_{1/2}}_{P_2}(s_0,a_2)$.
 Hence, neither of the two worst-case action-value functions dominates the other, so no robust action-value function exists under $\pi_{1/2}$, and \cref{assm:unique-robust-value} is not satisfied either.
\end{proof}